\documentclass[11pt]{article}%
\usepackage{float} 
\usepackage{fullpage}
\usepackage{amsfonts}
\usepackage{amsmath}
\usepackage{amssymb}
\usepackage{amsbsy}
\usepackage{amsthm}
\usepackage{mathtools}
\usepackage{epsfig}
\usepackage{graphicx}
\usepackage{colordvi}
\usepackage{graphics}
\usepackage{color}
\usepackage{bm}
\usepackage{verbatim} 
\numberwithin{equation}{section}

\newtheorem{theorem}{Theorem}[section]
\newtheorem{thm}{Theorem}[section]

\newtheorem{lemma}{Lemma}[section]

\newtheorem{prop}[thm]{Proposition}
\newtheorem{alg}[thm]{Algorithm}
\newtheorem{remark}[thm]{Remark}
\newtheorem{assumption}[thm]{Assumption}

\newcommand{\norm}[1]{\left\Vert#1\right\Vert}

\newcommand{\anorm}{\norm{\,\cdot \,}}
\newcommand{\forma}{( \,\cdot \, , \, \cdot \,)}
\newcommand{\te}{\tilde e}
\newcommand{\tx}{\tilde x}
\newcommand{\txa}{\tilde x^\alpha}
\newcommand{\xa}{x^\alpha}
\newcommand{\wa}{w^\alpha}
\newcommand{\grad}{\nabla}
\newcommand{\ran}{\text{Range\,}}
\newcommand{\nulll}{\text{Null\,}}
\DeclareMathOperator{\argmin}{argmin}
\newcommand{\eps}{\varepsilon}
\renewcommand{\div}{\text{div\,}}

\newcommand{\nr}[1]{\ensuremath{\left\|{#1}\right\|}}

\newcommand{\goto}{\rightarrow}

\newcommand{\bigo}{{\mathcal O}}

\newcommand{\f}{\frac}
\DeclareMathOperator{\dd}{\, d}

\begin{document}

\title{A proof that Anderson acceleration improves the convergence rate in linearly converging fixed point methods (but not in those converging quadratically)}

\author{
Claire Evans\thanks{School of Mathematical and Statistical Sciences, Clemson University, Clemson, SC 29634 (cevans4@g.clemson.edu)}
\and
Sara Pollock\thanks{Department of Mathematics, University of Florida, Gainesville, FL
  32611-8105 (s.pollock@ufl.edu), partially supported by NSF grant DMS1719849.}
\and
Leo G. Rebholz\thanks{School of Mathematical and Statistical Sciences, Clemson University, Clemson, SC 29634 (rebholz@clemson.edu), partially supported by NSF grant DMS1522191.}
\and
Mengying Xiao 
\thanks{Department of Mathematics, College of William \& Mary, Williamsburg VA 23187 (mxiao01@wm.edu).}
}
\date{}

\maketitle

\begin{abstract}
This paper provides the first proof that Anderson acceleration (AA) improves the convergence rate of general fixed point iterations. AA has been used for decades to speed up nonlinear solvers in many applications, however a rigorous mathematical justification of the improved convergence rate has remained lacking. The key ideas of the analysis presented here are relating the difference of consecutive iterates
to residuals based on performing the inner-optimization in a Hilbert space setting, and explicitly defining the gain in the optimization stage to be the ratio of improvement over a step of the unaccelerated fixed point iteration. The main result we prove is that AA improves the convergence rate of a fixed point  iteration to first order by a factor of the gain at each step. In addition to improving the convergence rate, our results indicate that AA increases the radius of convergence. Lastly, our estimate shows that while the linear convergence rate is improved, additional quadratic terms arise in the estimate, which shows why AA does not typically improve convergence in quadratically converging fixed point iterations. 
Results of several numerical tests are given which illustrate the theory.
\end{abstract}

\section{Introduction}
We study an acceleration technique for fixed point problems called Anderson 
acceleration, in which a history of search-directions is used to 
improve the rate of convergence of fixed-point iterations.
The method was originally introduced by D.G. Anderson in 1965
in the context of integral equations  \cite{anderson65}.  
It has recently been used in many applications, 
including multisecant methods for fixed-point iterations 
in electronic structure computations \cite{FaSa09}, geometry optimization problems
\cite{PDZGQL18}, various types of flow problems \cite{LWWY12,PRX18}, 
radiation diffusion and nuclear physics \cite{AJW17,TKSHCP15},
molecular interaction \cite{SM11}, machine learning \cite{GS18},
improving the alternating projections method for computing nearest correlation matrices 
\cite{HS16},
and on a wide range of nonlinear problems in the context of generalized minimal residual 
(GMRES) methods in \cite{WaNi11}. We further refer readers to 
\cite{K18,LW16,LWWY12,WaNi11} 
and references therein for detailed discussions on both practical implementation and 
a history of the method and its applications. 

Despite a long history of use and a strong recent interest, the first mathematical convergence results for 
Anderson acceleration (for both linear and nonlinear problems) appear in 2015 in \cite{ToKe15}, 
under the usual local assumptions for convergence of Newton iterations.  
However, this theory does not prove that Anderson acceleration improves 
the convergence of a fixed point iteration, or in other words accelerates convergence
in the sense of \cite{brezinski99}.  Rather, it proves that Anderson accelerated
fixed point iterations will converge in the neighborhood of a fixed point;
and, an upper bound on the convergence rate is shown to approach from above the 
convergence rate of the underlying fixed point iteration.  
While an important stage in the developing theory, this does not explain the 
efficacy of the method, which has gained popularity as practitioners  have continued 
to observe a dramatic speedup and increase in robustness from Anderson acceleration
over a wide range of problems.

The purpose of this paper is to address this gap in the theory by proving a rigorous estimate for 
Anderson acceleration that shows a guaranteed improvement in the convergence rate for 
fixed point iterations (for general $C^2$ functions) that converge linearly (with rate $\kappa$).  
By explicitly defining the gain of the optimization stage 
at iteration $k$ to be the ratio $\theta_k$ of the optimized objective function compared to that
of the usual fixed point method, we prove the new convergence rate is 
$\theta_k ( (1-\beta_{k-1}) + \beta_{k-1}\kappa)$ at step $k$, where 
$0 < \beta_{k-1} \le 1$ is a damping parameter and $\beta_{k-1}=1$ 
produces the undamped iteration.  
The key ideas to the proof are an expansion of the residual errors, 
developing expressions relating the difference of consecutive iterates and residuals, 
and explicitly factoring in the gain from the optimization stage.   
A somewhat similar approach is used by the
authors to prove that Anderson acceleration speeds up Picard iteration convergence for 
finite element discretizations 
of the steady Navier-Stokes equations in \cite{PRX18} (without the $C^2$ assumption
on the fixed-point operator), 
and herein we extend these ideas to general fixed point iterations.

In addition to the improved linear convergence rate, our analysis also indicates that Anderson acceleration 
introduces quadratic error terms, which is consistent with known results that Anderson acceleration does 
not accelerate quadratically converging fixed point methods (see the numerical experiments section below),
establishing a barrier which theoretically prevents establishing an
improved convergence rate for general fixed-point iterations.  
A third important result we show is that both Anderson acceleration
and the use of damping can extend the radius of convergence for the method, 
i.e. Anderson acceleration can allow the iteration to converge even when outside 
the domain where the fixed point function is contractive. 
An illustrative example of this is shown in \S \ref{subsec:nse}.

This paper is arranged as follows.  
In \S \ref{sec:anderson}, we review Anderson acceleration, describe the problem setting, 
and give some basic definitions and notation.  
\S \ref{sec:tech} gives several important technical results to make the later analysis 
cleaner and simpler.  
\S \ref{sec:c-rates} gives the main result of the paper, proving that the linear convergence rate is improved by Anderson acceleration, but additional quadratic error terms arise.  
\S \ref{sec:numerics} gives results from numerical tests, with the intent of illustrating the current contributions to the theory.  
Conclusions are given in the final section.
 
\section{Anderson acceleration}\label{sec:anderson}
In what follows, we will consider a fixed-point operator $g: X \goto X$
where $X$ is a  Hilbert space with norm $\anorm$ and inner-product $\forma$.  
The Anderson acceleration algorithm with depth $m$ applied to the fixed-point problem 
$g(x) = x$ reads as follows.

\begin{alg}[Anderson iteration] \label{alg:anderson}
The Anderson-acceleration with depth $m \ge 0$ and damping factors $0 < \beta_k \le 1$ 
reads: \\ 
Step 0: Choose $x_0\in X.$\\
Step 1: Find $\tilde x_1\in X $ such that $\tilde x_1 = g(x_0)$.  
Set $x_1 = \tilde x_1$. \\
Step $k$: For $k=1,2,3,\ldots$ Set $m_k = \min\{ k, m\}.$\\
\indent [a.] Find $\tilde x_{k+1} = g(x_k)$. \\
\indent [b.] Solve the minimization problem for $\{ \alpha_{j}^{k+1}\}_{k-m_k}^k$
\begin{align}\label{eqn:opt-v0}
\min_{\sum_{j=k-m_k}^{k} \alpha_j^{k+1}  = 1} 
\left\| \sum_{j=k-m_k}^{k} \alpha_j^{k+1}( \tx_{j+1} - x_j) \right\| .
\end{align}
\indent [c.] For damping factor $0 < \beta_k \le 1$, set
\begin{align}\label{eqn:update-v0}
x_{k+1} =  (1-\beta_k)\sum_{j= k-m_k}^k \alpha_j^{k+1} x_{j}
 + \beta_k \sum_{j= k-m_k}^k \alpha_j^{k+1} \tx_{j+1}.
\end{align}
\end{alg}

We will use throughout this work the stage-$k$ residual and error terms 
\begin{align}\label{eqn:update-err}
e_{k}  \coloneqq x_k - x_{k-1}, \quad
\tilde e_k  \coloneqq \tilde x_k - \tilde x_{k-1}, \quad
w_k  \coloneqq \tilde x_k - x_{k-1}. 
\end{align}
Define the following averages given by the solution 
$\alpha^{k+1} = \{\alpha_k^{k+1}\}_{j = k-m_k}^k$ to the optimization problem
\eqref{eqn:opt-v0} by 
\begin{align}\label{eqn:err-k}
x_k^\alpha = \sum_{j = k-m_k}^k \alpha^{k+1}_j x_j, \quad
\tx_{k+1}^\alpha = \sum_{j = k-m_k}^k \alpha^{k+1}_j \tx_{j+1}, \quad
w_{k+1}^\alpha = \sum_{j = k-m_k}^k \alpha^{k+1}_j (g(x_j) -x_j).
\end{align}
Then the update \eqref{eqn:update-v0} can be written in terms of the averages
$\xa$ and $\txa$,
\begin{align}\label{eqn:update-v1}
x_{k+1} = (1-\beta_k) x_k^\alpha + \beta_k \tx_{k+1}^\alpha,
\end{align}
and the stage-$k$ gain $\theta_k$ can be defined by
\begin{align}\label{eqn:thetak}
\nr{w_{k+1}^\alpha} = \theta_k \nr{w_{k+1}}. 
\end{align}
The key to showing the acceleration of this technique defined by taking a linear combination
of a history of steps corresponding to the coefficients of the optimization problem
\eqref{eqn:opt-v0} is connecting the gain $\theta_k$ given by \eqref{eqn:thetak} to
the differences of consecutive iterates and residual terms in \eqref{eqn:err-k}.
As such, the success (or failure) of the algorithm to reduce the residual is coupled to 
the success of the optimization problem at each stage of the algorithm.
As $\alpha_{k}^{k+1}= 1, \alpha^{k+1}_j = 0, j \ne k$ is an admissible solution to
\eqref{eqn:opt-v0}, it follows immediately that $0 \le \theta_k \le 1$.  
As discussed in the remainder, the improvement in the contraction rate of the fixed-point
iteration is characterized by $\theta_k$.

The two main components of the proof of residual convergence at an accelerated rate
are the expansion of the residual $w_{k+1}$ into $\wa_k$ and error terms
$e_{k-m_{k-1}}, \ldots, e_k$; and, control of the $e_j$'s in terms of the corresponding
$w_j$'s.  In the next section, the first of these is established for general $m$, 
and the second for the particular cases of depth $m=1$ and $m=2$, with the result
then extrapolated for general $m$.

\section{Technical preliminaries}\label{sec:tech}
There are two main technical results used in our theory.  The first is an 
expansion of the residual, and the second is a set of 
estimates relating the difference of consecutive iterates to residuals.
These are shown in \S \ref{subsec:resi-expansion} and \S \ref{subsec:ew}, 
respectively. The main results which depend on these estimates 
are then presented in \S \ref{sec:c-rates}.

For the bounds in \S\ref{subsec:ew} relating the difference of consecutive iterates
 to residuals, the operator $g: X \goto X$ is assumed Lipschitz continuous and contractive, 
as in \cite{PRX18}; see Assumption \ref{assume:contract}, below.
The results of \S \ref{subsec:resi-expansion} do not require the contractive
property, but require the assumption that $g$ is twice
continuously differentiable to allow for Taylor expansions of the error terms.
We denote the derivatives of $g$ by $g'(\cdot; \cdot)$ and $g''(\cdot; \cdot, \cdot)$,
and employ the standard notation that forms $g'(\cdot; \cdot)$ and $g''(\cdot; \cdot, \cdot)$
are linear with respect to the arguments to the right of the semicolon.
\begin{assumption}\label{assume:g} 
Let $X$ be a Hilbert space and $g: X \goto X$.  
Assume $g$ has a fixed point $x^\ast \in X$, 
and there are positive constants $\kappa$ and $\hat \kappa$ with
\begin{enumerate}
\item $g \in C^2(X)$.
\item $\nr{g'(y;u)} \le \kappa\nr{u}$ for each $y$ and all $u \in X$.
\item $\nr{g''(y;u,v)} \le \hat \kappa\nr{u}\nr{v}$ for each $y$ and all $u,v \in X$. 
\end{enumerate}
\end{assumption}
\begin{assumption}\label{assume:contract}
Let $X$ be a Hilbert space and $g: X \goto X$.  
Assume $\nr{g(y)-g(x)} \le \kappa \nr{x-y}$ for every $x,y\in X$, with $\kappa<1$.
\end{assumption}

By standard fixed-point theory, Assumption \ref{assume:contract} implies the existence
of a unique fixed-point $x^\ast$ of $g$ in $X$.
In a slight abuse of notation, the difference of consecutive iterates, 
$e_k = x_k - x_{k-1}$ is loosely referred to in this manuscript as an error term.  
As shown carefully in \cite{PRX18}, the true error $x_k - x^\ast$ is controlled
in norm by $e_j$, $j = k-m_k, \ldots, k$, for the depth $m$ algorithm so long as 
the coefficients from the optimization remain bounded.  
In the results of \S \ref{sec:c-rates}, the residual $w_k$ is shown to converge to zero
under Assumption \ref{assume:contract}. This is sufficient to establish convergence
of the error  $x_k - x^\ast$ to zero as
\[
\nr{x_k - x^\ast} \le \nr{x_k-g(x_k) } + \nr{g(x_k) - g(x^\ast)}
\le \nr{w_{k+1}} + \kappa \nr{x_k - x^\ast},
\]
by which $\nr{x_k - x^\ast} \le (1-\kappa)^{-1}\nr{w_{k+1}}$.

\subsection{Expansion of the residual}\label{subsec:resi-expansion}
Based on Assumption \ref{assume:g} the error term $\te_k$ of \eqref{eqn:err-k}
has a Taylor expansion
\begin{align}\label{eqn:T1}
\te_{k+1} \coloneqq g(x_{k}) - g(x_{k-1}) = \int_0^1 g'(z_k(t); e_k) \dd t,
\end{align}
where $z_{k}(t) =  x_{k-1}+te_k$.
For each $t \in [0,1]$ a second application of Taylor's Theorem provides
\begin{align}\label{eqn:T2}
g'(z_{k-1}(t);\cdot) = g'(z_k(t);\cdot) 
+ \int_0^1 g''(\hat z_{k,t}(s);z_{k-1}(t) - z_k(t),\cdot) 
\dd s, 
\end{align}
where $\hat z_{k,t}(s) =  z_{k-1}(t)+s(z_k(t) - z_{k-1}(t))$.
Using \eqref{eqn:T1}-\eqref{eqn:T2} we next derive an expansion of the
residual $w_{k+1}$ in terms of the differences of consecutive iterates
 $e_k, \ldots, e_{k-m_{k-1}}$. 
We start with the definition of the residual by \eqref{eqn:err-k} and 
the expansion of iterate $x_k$ by the update \eqref{eqn:update-v1}.
\begin{align}\label{eqn:wk001}
w_{k+1} = g(x_k) - x_k 
 = (1-\beta_{k-1}) (g(x_k) - \xa_{k-1}) + \beta_{k-1}(g(x_k) - \txa_k).
\end{align} 
Expanding the first term on the right hand side of \eqref{eqn:wk001} yields
\begin{align}\label{eqn:wk002}
g(x_k) - \xa_{k-1}& = \sum_{j = k-m_{k-1}-1}^{k-1} \alpha^k_j (g(x_k) - x_j)
\nonumber \\
& = \sum_{j = k-m_{k-1}-1}^{k-1} \alpha^k_j (g(x_j) - x_j)
+ \sum_{j = k-m_{k-1}}^{k} \left( \sum_{n = k-m_{k-1}-1}^{j-1}\alpha^k_n\right) 
(g(x_j) - g(x_{j-1}))
\nonumber \\
& = w_k^\alpha  
+ \sum_{j = k-m_{k-1}}^{k} \gamma_j \te_{j+1}, 
\end{align}
where 
\begin{align}\label{eqn:gamma}
\gamma_{j} \coloneqq \sum_{n = k-m_{k-1}-1}^{j-1} \alpha^k_n.
\end{align}
It is worth noting that $\gamma_{k} = 1$. 
Expanding the second term on the right hand side of \eqref{eqn:wk001}, we get
\begin{align}\label{eqn:wk003}
g(x_k) - \txa_{k} = \sum_{j = k-m_{k-1}-1}^{k-1} \alpha^k_j (g(x_k) - g(x_j))
= \sum_{j = k-m_{k-1}}^{k} \gamma_{j} \te_{j+1}.
\end{align}
Reassembling \eqref{eqn:wk001} with \eqref{eqn:wk002} and \eqref{eqn:wk003}
followed by \eqref{eqn:T1}, we have
\begin{align}\label{eqn:wk004}
w_{k+1}  = (1-\beta_{k-1}) w_k^\alpha + \sum_{j = k-m_{k-1}}^{k} \gamma_{j} \te_{j+1}
 = (1-\beta_{k-1}) w_k^\alpha +
  \sum_{j = k-m_{k-1}}^{k} \gamma_{j} \int_0^1 g'(z_j(t); e_j) \dd t.
\end{align}
We now take a closer look at the last term of \eqref{eqn:wk004}.
For each $j = k-m_{k-1},\ldots, k-1$, adding and subtracting intermediate averages
allows
\begin{align}\label{eqn:wk004int1}
\int_0^1 g'(z_j(t);e_j) \dd t = 
\int_0^1 g'(z_k(t);e_j) \dd t + 
\sum_{n=j}^{k-1} \int_0^1 g'(z_n(t);e_j) - g'(z_{n+1}(t); e_j) \dd t.
\end{align}
Applying now \eqref{eqn:T2} to each summand of \eqref{eqn:wk004int1} then yields
\begin{align}\label{eqn:wk004int2}
\int_0^1 g'(z_j(t);e_j) \dd t = 
\int_0^1 g'(z_k(t);e_j) \dd t + 
\sum_{n=j}^{k-1} \int_0^1 \int_0^1 g''(\hat z_{n+1,t}(s));z_n(t) - z_{n+1}(t),
  e_j ) \dd s\dd t.
\end{align}
Summing over the $j$'s after \eqref{eqn:wk004int2} is applied to each term,
the sum on the right hand side of \eqref{eqn:wk004} may be expressed as
\begin{align}\label{eqn:wk005}
\sum_{j = k-m_{k-1}}^{k} \!\! \gamma_{j} \int_0^1 g'(z_j(t);e_j) \dd t
& = 
\int_0^1 g'(z_k(t); \sum_{j = k-m_{k-1}}^k  \!\! \gamma_j e_j )\dd t 
\nonumber \\
&+ \sum_{j = k-m_{k-1}}^{k-1} 
  \int_0^1 \int_0^1\left( \sum_{n=j}^{k-1} 
  g''(\hat z_{n+1,t}(s);z_n(t) - z_{n+1}(t), \gamma_j e_j) \right)\dd s \dd t.
\end{align}
The next calculation shows that $\sum_{j = k-m_{k-1}}^k \gamma_j e_j $ 
is equal to $\beta_{k-1}w_k^\alpha$. 
First observe that $\gamma_{j} - \gamma_{j-1} = \alpha^{k}_{j-1}$ and 
$\gamma_{k-m_{k-1}} = \alpha^{k}_{k-m_{k-1}-1}$. 
Separating the first term of the sum and using $\gamma_{k}=1$,
\begin{align}\label{eqn:wk006}
\sum_{j=k-m_{k-1}}^k \gamma_{j}e_j 
&= x_k-x_{k-1} +  \sum_{j = k-m_{k-1}}^{k-1}\gamma_{j} (x_j-x_{j-1})
\nonumber \\
&= x_k-x_{k-1} + \gamma_{k-1}x_{k-1} - \sum_{j = k-m_{k-1}-1}^{k-2} \alpha_j^k x_j 
\nonumber \\
&= x_k-\alpha^k_{k-1}x_{k-1} - \sum_{j = k-m_{k-1}-1}^{k-2} \alpha_j^k x_j 
= x_k - \xa_{k-1}.
\end{align}
From \eqref{eqn:wk006} and the decomposition of  $x_k$ in terms of update 
\eqref{eqn:update-v0}, we have that
\begin{align}\label{eqn:wk007}
\sum_{j=k-m_{k-1}}^k \gamma_{j}e_j = 
x_k - \xa_{k-1} = (1-\beta_{k-1})x_{k-1}^\alpha + \beta_{k-1} \txa_k - \xa_{k-1}
= \beta_{k-1}(\txa_k - \xa_{k-1}) = \beta_{k-1} \wa_k.
\end{align}
Putting \eqref{eqn:wk007} together with \eqref{eqn:wk005} and \eqref{eqn:wk004}
then yields
\begin{align}\label{eqn:wk008}
w_{k+1} &= \int_0^1 (1-\beta_{k-1})\wa_k + \beta_{k-1} g'(z_k(t); \wa_k)  \dd t
\nonumber \\
&+ \sum_{j = k-m_{k-1}}^{k-1}  
  \int_0^1 \int_0^1\left( \sum_{n=j}^{k-1} 
  g''(\hat z_{n+1,t}(s);z_n(t) - z_{n+1}(t), \gamma_j e_j)\right) \dd s \dd t.
\end{align}
Based on the expansion of $w_{k+1}$ by \eqref{eqn:wk008} we now proceed to bound the 
higher order terms in the particular cases $m=1$ and $m=2$ to establish 
convergence of Algorithm \ref{alg:anderson} at an accelerated rate.

\subsection{Relating the difference of consecutive iterates to residuals}\label{subsec:ew}
We now derive estimates to bound (in norm) the $e_j$'s from the right hand side of 
\eqref{eqn:wk008} by the corresponding $w_j$'s.
The bounds in this subsection hold under Assumption \ref{assume:contract}, 
namely $g$ is a contractive operator.

Under Assumption \ref{assume:contract} we have the inequality
\begin{align}\label{eqn:ew001}
(1-\kappa)\nr{e_n} \le \nr{e_n} - \nr{\te_{n+1}} \le \nr{\te_{n+1} - e_{n}} = 
\nr{w_{n+1}-w_n}.
\end{align}

The next lemma establishes a bound for $e_{j-1}$ in terms of $w_j$ and $w_{j-1}$ in
the case of depth $m=1$.  The subsequent lemma generalizes the same idea for general $m$.
\begin{lemma}\label{lem:ew-m1}
Under the conditions of Assumption \ref{assume:contract}, the following bounds hold true:
\begin{align}\label{eqn:ew-m1a}
|\alpha^{j}_{j-1}|\nr{e_{j-1}} &\le \f{1}{1-\kappa} \nr{w_{j-1}},  
\\ \label{eqn:ew-m1b}
|\alpha^{j}_{j-2}|\nr{e_{j-1}} &\le \f{1}{1-\kappa} \nr{w_{j}}.
\end{align}
\end{lemma}
\begin{proof}
Begin by rewriting the optimization problem \eqref{eqn:opt-v0} in the equivalent form
\[
\eta = \argmin \nr{w_{j-1} + \eta (w_j - w_{j-1})}^2,
\]
where $\alpha^{j}_{j-1} = \eta$ and $\alpha^j_{j-2} = 1-\eta$.
The critical point $\eta$ then satisfies
$
\eta \nr{w_j - w_{j-1}}^2 = (w_{j-1}, w_j-w_{j-1}).
$
Applying Cauchy-Schwarz and triangular inequalities yields
$
|\eta| \nr{w_j - w_{j-1}} \le \nr{w_{j-1}}.
$
Applying \eqref{eqn:ew001} with $n = j-1$ yields the result \eqref{eqn:ew-m1a}.

Next, rewrite the optimization problem \eqref{eqn:opt-v0} in another equivalent form, 
\begin{align}\label{eqn:m1-gamma}
\gamma = \argmin \nr{w_{j} - \gamma (w_j - w_{j-1})}^2,
\end{align}
where the equivalence follows with $\alpha^j_{j-2} = \gamma$ and 
$\alpha^j_{j-1} = 1-\gamma$.  Following the same procedure as above yields
$|\gamma| \nr{w_{j} - w_{j-1}} \le \nr{w_j}$.  Applying \eqref{eqn:ew001} 
at level $n-1$ then yields the second result \eqref{eqn:ew-m1b}.
\end{proof}
The use of $\gamma$ as the second parameter of in the proof above is not purely 
coincidental, as this $\gamma$ agrees with the $\gamma^j_{j-1}$ used in 
\S \ref{subsec:resi-expansion}.
The same essential technique yields the necessary bounds for $m\ge2$. 
The estimate for general $m$ is given
in the lemma below, with the particular estimate for $m=2$ given as a proposition. 

As in the $m=1$ case above, two forms of the optimization problem are used.
The $\gamma$-formulation is used to bound the terms $\gamma_j\nr{e_j}$ that
appear from the expansion \eqref{eqn:wk008}; whereas, the $\eta$-formulation is
used to bound the terms $\nr{e_j}$ that appear in the numerator without leading 
optimization coefficients.  It is then of particular importance that estimates of the
form $c \nr{e_j} \le \Sigma k_n\nr{w_n}$ have the property that $c$ is bounded away 
from zero.
This is a reasonable assumption on the leading coefficient $c=\alpha^k_{k-1}$ for
each $k$, as some nonvanishing component in the latest search direction 
is necessary for progress.  It is also a reasonable assumption on 
$c = 1-\alpha^k_{k-m_{k-1}-1}$, meaning the coefficient of the earliest search 
direction considered is bounded away from unity.  Presumably, 
$|\alpha^k_{k-m_{k-1}-1}|<1$ is a reasonable assumption to make, although this is
not explicitly required ({\em cf.,} \cite{PRX18}).

\begin{lemma}\label{lem:ew-gm}
Under the conditions of Assumption \ref{assume:contract}, the following bounds hold true:
\begin{align}\label{eqn:gm-jeta}
|\alpha^j_{j-1}\nr{e_{j-1}} &\le \f{1}{1-\kappa}
\left(|\eta_{j-1}|\nr{w_{j-1}}
+ \sum_{n = j-{m_{j-1}}-1}^{j-2} |\alpha_{n-1}^j| \nr{w_n}\right)
\\ \label{eqn:gm-jmeta}
|1-\alpha^j_{j-{m_{j-1}-1}}|\nr{e_{j-m_{j-1}}} & \le \f{1}{1-\kappa}\left(
  \sum_{n = j-{m}+2}^{j} |\alpha_{n-1}^j| \nr{w_n}
+ |\eta_{j-m+2}|\nr{w_{j-m+1}} + \nr{w_{j-m}}
\right)
\\ \label{eqn:gm-ngam}
|\gamma_{p-1}|\nr{e_{p-1}} & \le \f{1}{1-\kappa} \left( 
\sum_{n = j-m_{j-1}}^{p-2} |\alpha^j_{n-1}| \nr{w_n}
+ |\gamma_{p-2}|\nr{w_{p-1}}+ |\gamma_{p}| \nr{w_{p}} \right.
\nonumber \\
&+ \left. \sum_{n = p+1}^{j} |\alpha^j_{n-1}|\nr{w_n}\right). 
\end{align}
with $\eta_{j-1} = \alpha^j_{j-1} + \alpha^j_{j-2}$ as in \eqref{eqn:mg-eta}, and 
$\gamma_p$,$\gamma_{p-1}$,$\gamma_{p-2},$ given below by \eqref{eqn:mg-gamma}.
\end{lemma}
\begin{proof}
The optimization problem \eqref{eqn:opt-v0} at level $j$ is to minimize
\[
\nr{\sum_{n = j-m_{j-1}-1}^{j-1} \alpha^j_n w_{n+1}} \ 
~\text{subject to } \sum_{n = j-m_{j-1}-1}^{j-1}\alpha^j_n = 1.
\]
Differencing from the left and right respectively, this can be posed as 
the following unconstrained optimization problems:
\begin{align}\label{eqn:mg-eta}
\text{ minimize }& \nr{ w_{j-m_{j-1}} 
+ \sum_{n = j-m_{j-1}+1}^j \eta_n (w_n - w_{n-1})}^2, 
& \eta_n &=  \sum_{i= n-1}^{j-1} \alpha_i^j.
\\ \label{eqn:mg-gamma}
\text{ minimize }& \nr{ w_{j} - \sum_{n = j-m_{j-1}+1}^{j} \gamma_{n-1} 
(w_n - w_{n-1})}^2,
& \gamma_n &= \sum_{i = j-m_{j-1}-1}^{n-1}\alpha_i^j.
\end{align}
Note that \eqref{eqn:mg-gamma} coincides with \eqref{eqn:gamma} which agrees
with the unconstrained form of the optimization problem in for instance \cite{FaSa09}.
{To help reduce notation,} denote $m = m_{j-1}$ for the remainder of the proof.

Starting with estimate \eqref{eqn:gm-jeta} we are concerned with bounding in norm the 
leading term difference term $w_j - w_{j-1}$.
Expanding the norm squared \eqref{eqn:mg-eta} as an inner-product and seeking the 
critical point for $\eta_j$ yields
\[
\eta_j\nr{w_j - w_{j-1}}^2 + (w_j - w_{j-1},w_{j-m}) 
+ \sum_{n = j-m+1}^{j-1} \eta_n(w_j - w_{j-1}, w_n - w_{n-1}) = 0.
\]
Recombining the terms inside the sum, noting 
$\eta_{n-1} - \eta_n = \alpha^j_{n-2}$, and $\eta_j = \alpha^j_{j-1}$ obtain
\[
\alpha^j_{j-1}\nr{w_j - w_{j-1}}^2 = 
-(\alpha^j_{j-1} + \alpha^j_{j-2}) (w_j - w_{j-1},w_{j-1})  
- \sum_{n = j-m}^{j-2} \alpha^j_{n-1} (w_j - w_{j-1}, w_n).
\]
Applying Cauchy-Schwarz and triangle inequalities then yields
\[
|\alpha^j_{j-1}|\nr{w_j - w_{j-1}} \le |\alpha^j_{j-1} + \alpha^j_{j-2}|\nr{w_{j-1}}
+ \sum_{n = j-{m}}^{j-2} \alpha_{n-1}^j \nr{w_n} .
\]
Applying \eqref{eqn:ew001}, the result \eqref{eqn:gm-jeta} follows.

Following the same idea for estimate \eqref{eqn:gm-jmeta}, we are now concerned with
bounding in norm the final difference term $w_{j-m+1} - w_{j-m}$. Again expanding
\eqref{eqn:mg-eta} as an inner-product and seeking the critical point this time
for $\eta_{j-m+1}$ yields
\[
\eta_{j-m+1}\nr{w_{j-m+1} - w_{j-m}}^2 + (w_{j-m+1} - w_{j-m},w_{j-m}) 
+ \sum_{n = j-m+2}^{j} \eta_n(w_{j-m+1} - w_{j-m}, w_n - w_{n-1}) = 0.
\]
Recombining terms noting $\eta_{j-m+1} = 1-\alpha^j_{j-m-1}$
\begin{align*}
(1-\alpha^j_{j-m-1})\nr{w_{j-m+1} - w_{j-m}}^2 &= 
 \sum_{n = j-m+2}^{j} \alpha^j_{n-1} (w_{j-m+1} - w_{j-m}, w_n) 
\nonumber \\
& - (w_{j-m+1} - w_{j-m}, \eta_{j-m+2}w_{j-m+1} +  w_{j-m}).
\end{align*}
Applying Cauchy-Schwarz and triangle inequalities then yields
\[
|1-\alpha^j_{j-m-1}|\nr{w_{j-m+1} - w_{j-m}} 
\le \left(\sum_{n = j-{m}+2}^{j} |\alpha_{n-1}^j| \nr{w_n}\right) 
+ |\eta_{j-m+2}|\nr{w_{j-m+1}} + \nr{w_{j-m}}.
\]
The result \eqref{eqn:gm-jmeta} follows by \eqref{eqn:ew001}.

Similarly for \eqref{eqn:gm-ngam}, 
expanding the norm of \eqref{eqn:mg-gamma} as an inner product and seeking
the critical point for each $\gamma_p$ yields
\[
\gamma_{p-1}\nr{w_p - w_{p-1}}^2 = (w_p - w_{p-1},w_j) 
- \sum_{n = j-m+1, n \ne p}^{j} \gamma_{n-1} (w_p - w_{p-1},w_n-w_{n-1}).
\]
Recombining the terms inside the sum using $\gamma_n - \gamma_{n-1} = \alpha^j_{n-1}$, 
and $\gamma_{j-m} = \alpha^j_{j-m-1}$, we obtain
\begin{align*}
\gamma_{p-1}\nr{w_p - w_{p-1}}^2 &= 
\sum_{n = j-m}^{p-2} \alpha^j_{n-1}(w_p-w_{p-1},w_n) 
- \gamma_{p-2}(w_p-w_{p-1}, w_{p-1})+ \gamma_{p}(w_p-w_{p-1}, w_{p})
\nonumber \\
&+\sum_{n = p+1}^{j} \alpha^j_{n-1}(w_p-w_{p-1},w_n).
\end{align*}
Applying now Cauchy-Schwarz and triangle inequalities,
\begin{align*}
|\gamma_{p-1}|\nr{w_p - w_{p-1}} &\le 
\sum_{n = j-m}^{p-2} |\alpha^j_{n-1}| \nr{w_n}
+ |\gamma_{p-2}|\nr{w_{p-1}}+ |\gamma_{p}| \nr{w_{p}}
+\sum_{n = p+1}^{j} |\alpha^j_{n-1}|\nr{w_n}. 
\end{align*}
Applying \eqref{eqn:ew001}, the result \eqref{eqn:gm-ngam} follows.
\end{proof}

For the convenience of subsequent calculations, the bounds \eqref{eqn:mg-eta}
and \eqref{eqn:mg-gamma} used to bound $\nr{w_{k+1}}$ for the case of depth $m=2$
are summarized in the following proposition. 
\begin{prop}[Depth $m=2$]\label{prop:ew-m2} 
With depth $m=2$ the estimates \eqref{eqn:gm-jeta} and \eqref{eqn:gm-ngam} reduce to
\begin{align}\label{eqn:m2-jeta}
|\alpha^j_{j-1}|\nr{e_{j-1}} &\le \f{1}{1-\kappa}\left(
|\alpha^j_{j-1} + \alpha^j_{j-2}|\nr{w_{j-1}} + |\alpha_{j-3}^j| \nr{w_{j-2}} \right)
\\ \label{eqn:m2-jmeta}
|1-\alpha^j_{j-2}|\nr{e_{j-2}} & \le \f{1}{1-\kappa}\left(
|\alpha^j_{j-1}|\nr{w_j} + |\alpha^j_{j-1}|\nr{w_{j-1}} + \nr{w_{j-2}} \right)
\\ \label{eqn:m2-jgam}
|\gamma_{j-1}|\nr{e_{j-1}} & \le \f{1}{1-\kappa}\left(
\nr{w_j} + |\alpha^j_{j-3}|\nr{w_{j-1}} + |\alpha^j_{j-3}|\nr{w_{j-2}} \right)
\\ \label{eqn:m2-j1gam}
|\gamma_{j-2}|\nr{e_{j-2}} & \le \f{1}{1-\kappa} \left(
|\alpha^j_{j-1}\nr{w_j} + |1-\alpha^j_{j-1}|\nr{w_{j-1}} \right).
\end{align}
\end{prop}
The second two bounds \eqref{eqn:m2-jgam} and \eqref{eqn:m2-j1gam} follow from 
\eqref{eqn:gm-ngam} noting from \eqref{eqn:mg-gamma}, 
that for $m=2$ we have
$\gamma_{j-2} = \alpha_{j-3}^j$, $\gamma_{j-1} = 1- \alpha_{j-1}^j$ and 
$\gamma_{j} = 1$.
The approach taken in \cite{PRX18} is to reduce the right hand side of 
\eqref{eqn:m2-jmeta} and \eqref{eqn:m2-jgam} to two terms each by relating their
 expansion to that of \eqref{eqn:m2-jeta} and \eqref{eqn:m2-j1gam}, respectively. 
Here the terms are left as they are to emphasize the direct generality to greater 
depth $m$.
\subsection{Explicit computation of the optimization gain} \label{subsec:optgain}
The stage-$k$ gain $\theta_k$ has a simple description assuming the optimization is
performed over a norm $\anorm$ induced by an inner product $\forma$, 
in other words in a Hilbert space setting.

Consider the unconstrained $\gamma$-form of the optimization problem 
\eqref{eqn:mg-gamma} at iteration $k$ with depth $m$: Find 
$\gamma_{k-m+1}, \ldots, \gamma_{k}$ that minimize
\begin{align}\label{eqn:gamma-opt}
\nr{ w_{k+1} - \sum_{n = k-m+1}^{k} \gamma_{n} (w_{n+1} - w_{n})}^2 
=
\nr{ w_{k+1} - F^k \gamma^k}^2, 
\end{align}
Where $F$ is the matrix with columns $w_{n+1}-w_{n}$, $n = k-m+1, \ldots, k$ and 
$\gamma^k$ is the corresponding vector of coefficients 
$\gamma_{k-m+1}, \ldots, \gamma_{k}$.  Indeed, \eqref{eqn:gamma-opt} 
(or equivalently reindexed)
is the preferred way to state the optimization problem \cite{WaNi11}, particularly
in the case where $\anorm$ is the $l_2$ norm and a fast $QR$ algorithm can be used.

This is also the preferred statement of the problem to understand the gain $\theta_k$
from \eqref{eqn:thetak}, which satisfies $\nr{\wa_{k+1}} = \theta_k\nr{w_{k+1}}$
Define the unique decomposition $w_{k+1} = w_R + w_N$ with 
$w_R \in \ran(F^k)$ and $w_N \in \nulll((F_k)^T)$.
Then $w_N$ is the least-squares residual satisfying 
$\nr{w_N} = \nr{w_{k+1} - F^k \delta^k} = \nr{\wa_{k+1}} = \theta_k \nr{w_{k+1}}$
meaning
\begin{align}
\theta_{k} = \sqrt{1 - \f{\nr{w_R}^2}{\nr{w_{k+1}}^2}},
\end{align}
and, $\theta_k$ has the interpretation of the direction-sine between $w_{k+1}$ and
the subspace spanned by $\{w_{n+1}-w_n\}_{n  = k-m+1}^k$. 
This is particularly clear in the case $m=1$ where by solving for the critical
point $\gamma$ of \eqref{eqn:m1-gamma} yields
\[
\gamma = \f{(w_{k+1},w_{k+1}-w_k)}{\nr{w_{k+1}-w_k}^2}.
\]
Expanding $\theta_k^2 \nr{w_{k+1}}^2= \nr{w_{k+1} - \gamma({w_{k+1}-w_k})}^2$ and using
the particular value of $\gamma$ above yields
\[
1-\theta_k^2 = \f{(w_{k+1},w_{k+1}-w_k)^2}{\nr{w_{k+1}-w_k}^2\nr{w_{k+1}}^2},
\]
with the clear interpretation that $(1-\theta_k^2)^{1/2}$ 
is the direction cosine between $w_{k+1}$ and $w_{k+1}-w_k$, 
hence $\theta_k$ is the direction-sine.

If indeed an (economy) $QR$ algorithm $F^k = Q_1 R_1$ is used to solve the optimization
problem then $\theta_k = \sqrt{1 - (\nr{Q_1^Tw_{k+1}}/\nr{w_{k+1}})^2}$, which 
can be used to predict whether an accelerated step would be (sufficiently) beneficial.
This explicit computation of $\theta_k$ is used in \S \ref{subsec:qldamp} 
to propose an adaptive damping strategy based on the gain at each
step.  Finally, it is noted that the improvement in the gain $\theta_k$ as 
$m$ is increased depends on sufficient linear independence or small direction
cosines between the columns of $F^k$, as information from earlier in the history is 
added. This is discussed in some greater depth in \cite{WaNi11}.

\section{Convergence rates for depths $m=1$ and $m=2$}\label{sec:c-rates}
First we put the expansion \eqref{eqn:wk008} together with the bounds 
\eqref{eqn:ew-m1a}-\eqref{eqn:ew-m1b} for a 
convergence proof for the simplest case of $m=1$.

\begin{theorem}[Convergence of the residual with depth $m=1$]\label{thm:m1}
On satisfaction of Assumptions \ref{assume:g} and \ref{assume:contract}, 
if the coefficients $\alpha^{k+1}_k,\alpha^k_{k-1}$ remain bounded and bounded 
away from zero, 
the following bound holds for the residual $w_{k+1}$ from Algorithm
\ref{alg:anderson} with depth $m=1$:
\[
\nr{w_{k+1}} \le \theta_k((1-\beta_{k-1}) + \kappa \beta_{k-1})\nr{w_k} +
\bigo\left(\nr{w_k}^2 \right) 
+\bigo\left(\nr{w_{k-1}}^2 \right).
\]
\end{theorem}
\begin{remark}
The assumptions on the coefficients $\alpha_j^k$ arising from the optimization problem are 
similar to those of \cite{ToKe15}.
These assumptions could be eliminated by solving instead a constrained optimization 
problem that enforces boundedness of the parameters, 
resulting in a modified gain $\hat \theta_k$ which satisfies
$\theta_k \le \hat \theta_k \le 1$.
\end{remark}
\begin{proof}
In this case the expansion found for $w_{k+1}$ in \eqref{eqn:wk008} reduces to
\begin{align}\label{eqn:m1001}
w_{k+1} &= \int_0^1 (1-\beta_{k-1})\wa_k + \beta_{k-1} g'(z_k(t); \wa_k)  \dd t
\nonumber \\
&+  \int_0^1 \int_0^1g''(\hat z_{k,t}(s);z_{k-1}(t) - z_{k}(t), \gamma_{k-1} e_{k-1}) 
\dd s \dd t.
\end{align}
Taking norms of both sides and applying Assumption \ref{assume:g}, 
\eqref{eqn:thetak} and the triangle inequality,
\begin{align}\label{eqn:m1002}
\nr{w_{k+1}} \le \theta_k ((1-\beta_{k-1}) + \kappa\beta_{k-1}) \nr{w_k} 
+ \hat \kappa (\nr{e_k} + \nr{e_{k-1}})\gamma_{k-1}\nr{e_{k-1}}.
\end{align}
The preceding bound \eqref{eqn:m1002} holds regardless of whether $g$ is globally
contractive (Assumption \ref{assume:contract}), 
hence for error terms $\nr{e_k}$ and $\nr{e_{k-1}}$ small 
enough, contraction of the error may be observed depending on the search direction,
particularly if a damping factor $0 < \beta < 1$ is 
applied, and if the gain $\theta_k$ is sufficiently less than one.  This justifies
the observation that Anderson acceleration can enlarge the effective domain of 
convergence of a fixed point iteration.

For the remainder of the calculation, we consider the case of a contractive operator, 
meaning Assumption \ref{assume:contract} is satisfied. Applying \eqref{eqn:ew-m1b} 
with $j = k$ to the $\gamma_{k-1}\nr{e_{k-1}}$, recalling by \eqref{eqn:gamma} we have 
$\gamma_{k-1} = \alpha^k_{k-2}$; and, applying \eqref{eqn:ew-m1a} with $j = k+1$ 
and $j=k$ respectively to the remaining $\nr{e_k}$ and $\nr{e_{k-1}}$ allows
\begin{align}\label{eqn:m1003}
\nr{w_{k+1}} &\le \theta_k ((1-\beta_{k-1}) + \kappa \beta_{k-1}) \nr{w_k} 
+ \f{\hat \kappa}{(1-\kappa)^2}\left(\f{\nr{w_k}}{\alpha^{k+1}_k} 
+ \f{\nr{w_{k-1}}}{\alpha^k_{k-1}}\right)\nr{w_k}
\nonumber \\
& =  \theta_k ((1-\beta_{k-1}) + \kappa\beta_{k-1}) \nr{w_k} 
+ \bigo \left( \nr{w_k}^2 \right)
+ \bigo \left( \nr{w_{k-1}}^2 \right).
\end{align}
\end{proof}
As discussed in \S \ref{subsec:ew}, $\alpha_k^{k+1}$ and $\alpha_{k-1}^k$ are
each the leading coefficients in their respective optimization problems, multiplying
the most recent iterate.  As such, these coefficients may be reasonably considered
bounded away from zero.

The case of $m=2$ follows similarly, combining \eqref{eqn:wk008} with 
\eqref{eqn:m2-jeta}-\eqref{eqn:m2-j1gam}.
\begin{theorem}[Convergence of the residual with depth $m=2$]\label{thm:m2}
On satisfaction of Assumptions \ref{assume:g} and \ref{assume:contract}, 
if the coefficients $\alpha^k_{k-3}, \dots, \alpha^{k}_{k-1}$ remain bounded, 
and $\alpha^k_{k-1}$ and $1-\alpha^k_{k-3}$ remain bounded away from zero, 
the following bound holds for the residual $w_{k+1}$ from Algorithm
\ref{alg:anderson} with depth $m=2$.
\[
\nr{w_{k+1}} \le \theta_k((1-\beta_{k-1}) + \kappa \beta_{k-1})\nr{w_k} +
\bigo\left(\nr{w_k}^2 \right) 
+\bigo\left(\nr{w_{k-1}}^2 \right) 
+\bigo\left(\nr{w_{k-2}}^2 \right).
\]
\end{theorem}
\begin{proof}
For depth $m=2$ the residual expansion \eqref{eqn:wk008} reduces to
\begin{align*}
w_{k+1} &= \int_0^1 (1-\beta_{k-1})\wa_k + \beta_{k-1} g'(z_k(t); \wa_k)  \dd t
\nonumber \\
& + \int_0^1 \int_0^1 
  g''(\hat z_{k,t}(s);z_{k-1}(t) - z_{k}(t), \gamma_{k-1} e_{k-1}) \dd s \dd t.
\nonumber\\
& + \int_0^1 \int_0^1 
  g''(\hat z_{k-1,t}(s);z_{k-2}(t) - z_{k-1}(t), \gamma_{k-2} e_{k-2}) \dd s \dd t.
\nonumber \\
& + \int_0^1 \int_0^1 
  g''(\hat z_{k,t}(s);z_{k-1}(t) - z_{k}(t), \gamma_{k-2} e_{k-2}) \dd s \dd t.
\end{align*}
Taking norms of both sides and applying \eqref{eqn:thetak} and the triangle inequality,
\begin{align}\label{eqn:m2002}
\nr{w_{k+1}}& \le \theta_k((1-\beta_{k-1}) + \kappa \beta_{k-1})\nr{w_k} +
\hat \kappa(\nr{e_k} + \nr{e_{k-1}})|\gamma_{k-1}|\nr{e_{k-1}} .
\nonumber \\
& + \hat \kappa (\nr{e_{k-2}} + 2\nr{e_{k-1}} + \nr{e_{k}})|\gamma_{k-2}|\nr{e_{k-2}}.
\end{align}
Applying \eqref{eqn:m2-jgam} and \eqref{eqn:m2-j1gam} to \eqref{eqn:m2002} yields 
\begin{align}\label{eqn:m2002a}
\nr{w_{k+1}}& \le \theta_k((1-\beta_{k-1}) + \kappa \beta_{k-1})\nr{w_k} 
\nonumber \\
&+\f{\hat \kappa}{1-\kappa}(\nr{e_k} + \nr{e_{k-1}})
\left(\nr{w_k} + |\alpha^k_{k-3}|\nr{w_{k-1}} + |\alpha^k_{k-3}|\nr{w_{k-2}} \right)
\nonumber \\
& + \f{\hat \kappa}{1-\kappa} 
(\nr{e_{k}} + 2\nr{e_{k-1}} + \nr{e_{k-2}})
\left(|\alpha^k_{k-1}\nr{w_k} + |1-\alpha^k_{k-1}|\nr{w_{k-1}} \right).
\end{align}
Applying \eqref{eqn:m2-jeta} with $j=k+1$ and $j=k$ together with
\eqref{eqn:m2-jmeta} to \eqref{eqn:m2002a} then yields
\begin{align}\label{eqn:m2003}
\nr{w_{k+1}}& \le \theta_k((1-\beta_{k-1}) + \kappa \beta_{k-1})\nr{w_k} 
\nonumber \\
&+ \f{\hat \kappa}{(1-\kappa)^2}\Bigg(\f{1}{\alpha^{k+1}_k}\left(
|\alpha^{k+1}_k + \alpha^{k+1}_{k-1}|\nr{w_k} + |\alpha^{k+1}_{k-2}|\nr{w_{k-1}}\right)
\nonumber \\
&+\f{1}{\alpha^{k}_{k-1}}\left(
|\alpha^{k}_{k-1} + \alpha^{k}_{k-2}|\nr{w_{k-1}} + |\alpha^{k}_{k-3}|\nr{w_{k-2}}\right)
\Bigg) 
\nonumber \\
& \times
\left(\nr{w_k} + |\alpha^k_{k-3}|\nr{w_{k-1}} + |\alpha^k_{k-3}|\nr{w_{k-2}} \right)
\nonumber \\
& + \f{\hat \kappa}{(1-\kappa)^2}\Bigg(
 \f{1}{\alpha^{k+1}_k}\left(
|\alpha^{k+1}_k + \alpha^{k+1}_{k-1}|\nr{w_k} + |\alpha^{k+1}_{k-2}|\nr{w_{k-1}}\right)
\nonumber \\
&+\f{1}{\alpha^{k}_{k-1}}\left(
|\alpha^{k}_{k-1} + \alpha^{k}_{k-2}|\nr{w_{k-1}} + |\alpha^{k}_{k-3}|\nr{w_{k-2}}\right)
\nonumber \\
& +\f{1}{1-\alpha^k_{k-3}}
\left(|\alpha^k_{k-1}|\nr{w_k} + |\alpha^k_{k-1}|\nr{w_{k-1}} + \nr{w_{k-2}} \right)
\Bigg) 
\nonumber \\
& \times
\left(|\alpha^k_{k-1}\nr{w_k} + |1-\alpha^k_{k-1}|\nr{w_{k-1}} \right).
\end{align}
And, \eqref{eqn:m2003} satisfies
\begin{align}\label{eqn:m2004}
\nr{w_{k+1}}& \le \theta_k((1-\beta_{k-1}) + \kappa \beta_{k-1})\nr{w_k} +
\bigo\left(\nr{w_k}^2 \right) 
+\bigo\left(\nr{w_{k-1}}^2 \right) 
+\bigo\left(\nr{w_{k-2}}^2 \right),
\end{align}
where the higher order terms have bounded coefficients.
\end{proof}
\begin{remark}
To avoid the extra assumption that $|1-\alpha^k_{k-3}|$ remains bounded away from 
zero, the term $\nr{e_{k-2}}$ of \eqref{eqn:m2002} could be bounded instead by 
\eqref{eqn:gm-ngam} with $j = k-1$, by which \eqref{eqn:m2004} is replaced by 
\begin{align*}
\nr{w_{k+1}}& \le \theta_k((1-\beta_{k-1}) + \kappa \beta_{k-1})\nr{w_k} 
\nonumber \\
&+\bigo\left(\nr{w_k}^2 \right) 
+\bigo\left(\nr{w_{k-1}}^2 \right) 
+\bigo\left(\nr{w_{k-2}}^2 \right)
+\bigo\left(\nr{w_{k-3}}^2 \right).
\end{align*}
Moreover this generalizes to higher order.
\end{remark}
Finally, we state without proof the general result which can be extrapolated from 
\eqref{eqn:wk008} and \eqref{eqn:gm-jeta}-\eqref{eqn:gm-ngam} as was done explicitly
for depth $m=2$, above.
\begin{prop}\label{prop:gm}
On satisfaction of Assumptions \ref{assume:g} and \ref{assume:contract}, 
if the coefficients $\alpha^k_{k-m-1}, \dots, \alpha^{k}_{k-1}$ remain bounded, 
and $\alpha^k_{k-1}$ and $1-\alpha^k_{k-m-1}$ remain bounded away from zero, 
the following bound holds for the residual $w_{k+1}$ from Algorithm
\ref{alg:anderson} with depth $m$.
\[
\nr{w_{k+1}} \le \theta_k((1-\beta_{k-1}) + \kappa \beta_{k-1})\nr{w_k} +
\sum_{j = 0}^{m}\bigo\left(\nr{w_{k-j}}^2 \right). 
\]
\end{prop}
As discussed in \S \ref{subsec:optgain}, even as the higher order terms accumulate, there is
still an advantage to some extent to considering greater depth $m$, due to the 
improved gain from the optimization problem.
However, in practice this must be weighed against
the computational cost of raising $m$ (which can become significant) 
and the accuracy of one's optimization solver.  
In our tests, little improvement is found past $m=3$.

\section{Numerical tests}\label{sec:numerics}

We now give results of several numerical tests that illustrate the theory above.  In particular, we illustrate that Anderson speeds up  linear convergence, slows down quadratic convergence, and increases the radius of 
convergence in agreement with the presented theory.  It is not our purpose in this
section to show how well Anderson acceleration works on a wide variety of problems;
for this, see the references in the introduction.  

\subsection{Simple illustrative tests for the scalar case}\label{subsec:scalar-num}

We start with results of some simple tests for scalar problems, 
which illustrate the theory above.  For scalar fixed point iterations, 
it only makes sense to consider Anderson for $m=1$, 
since one can solve explicitly for the optimization parameter that makes 
the objective function zero, hence $\theta_k=0$ at each step.  
We take $\beta_k=1$ in each of these 1D tests.  
We remark that for the 1D case with $m=1$, Anderson acceleration of the fixed point problem with $g(x)$ is equivalent to the secant method applied to $f(x)=g(x)-x$ (this follows from \cite{FaSa09} but could also be easily shown by writing out the methods), but still feel it is instructive to show these simple tests.

The fixed point iterations we consider are:
\begin{align*}
FPP_1: \ \ \ x_{k+1} = g_1(x_k) & = 1 + \frac{2}{x_k}, & x_0=2.1, \\
FPP_2: \ \ \ x_{k+1}  = g_2(x_k) & = x_k - \frac{\cos(x_k) - \sin(x_k)}{-\sin(x_k) - \cos(x_k)}, & x_0=1, \\
FPP_3: \ \ \ x_{k+1}  = g_3(x_k)  &= x_k^2 - 2, & x_0=4.
\end{align*}

Results from these iterations, with ($m=1$) and without ($m=0$) Anderson acceleration are shown in Figure \ref{tests1Dc}.  For $FPP_1$ with $m=0$ we expect and observe linear convergence with a rate of $|g'(2)|=0.5$ to $x^*=2$, but with $m=1$ the convergence becomes superlinear.  Since $\theta_k=0$, our theory shows that error then depends only on quadratic terms, which is consistent with these results.

$FPP_2$ is the Newton iteration for finding the zero of $f(x)=\cos(x)-\sin(x)$, and the fixed point the method converges to is $x^*=\frac{\pi}{4}$.  Since here the $m=0$ test is Newton's method with a smooth $g$ and good initial guess, the convergence is expected and observed to be quadratic.  With $m=1$, we see convergence is slightly worse, which agrees with the theory above: Anderson acceleration adds additional quadratic terms to the residual, which are significant in a quadratically converging iteration.

Lastly in 1D, we consider $FPP_3$, which for $m=0$ is not expected to converge to $x^*=2$ when $x_0>2$ since $g_3$ is not contractive near the fixed point ($g'(2)=4$).  As expected, with $m=0$, the iteration grows exponentially and by iteration 4 has reached a value of $10^{10}$.  However, with $m=1$ the convergence radius is increased (from 0) to be large enough that the iteration converges even with $x_0=4$.

\begin{figure}[!ht]
	\centering
	\includegraphics[width = .48\textwidth,height=.25\textwidth]{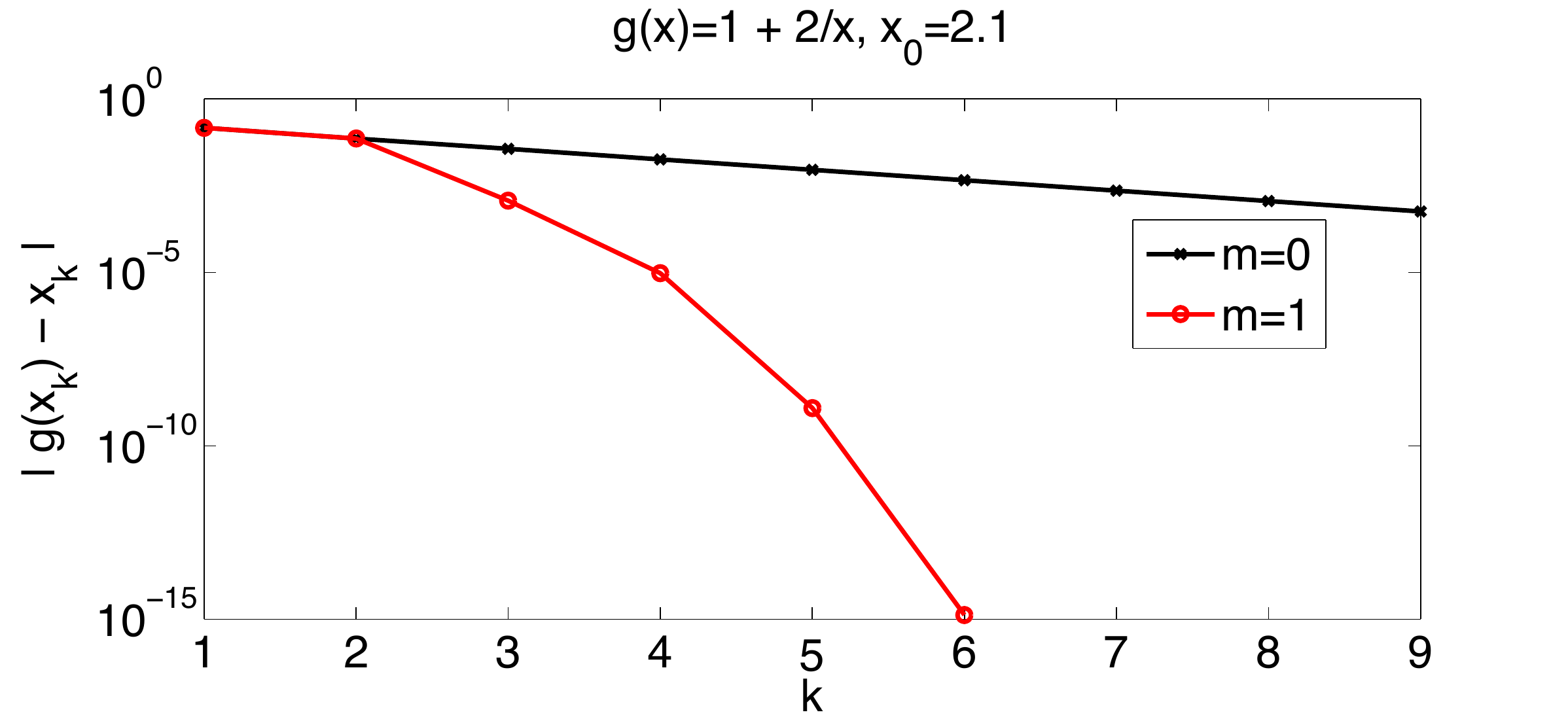}	
	\includegraphics[width = .48\textwidth,height=.25\textwidth]{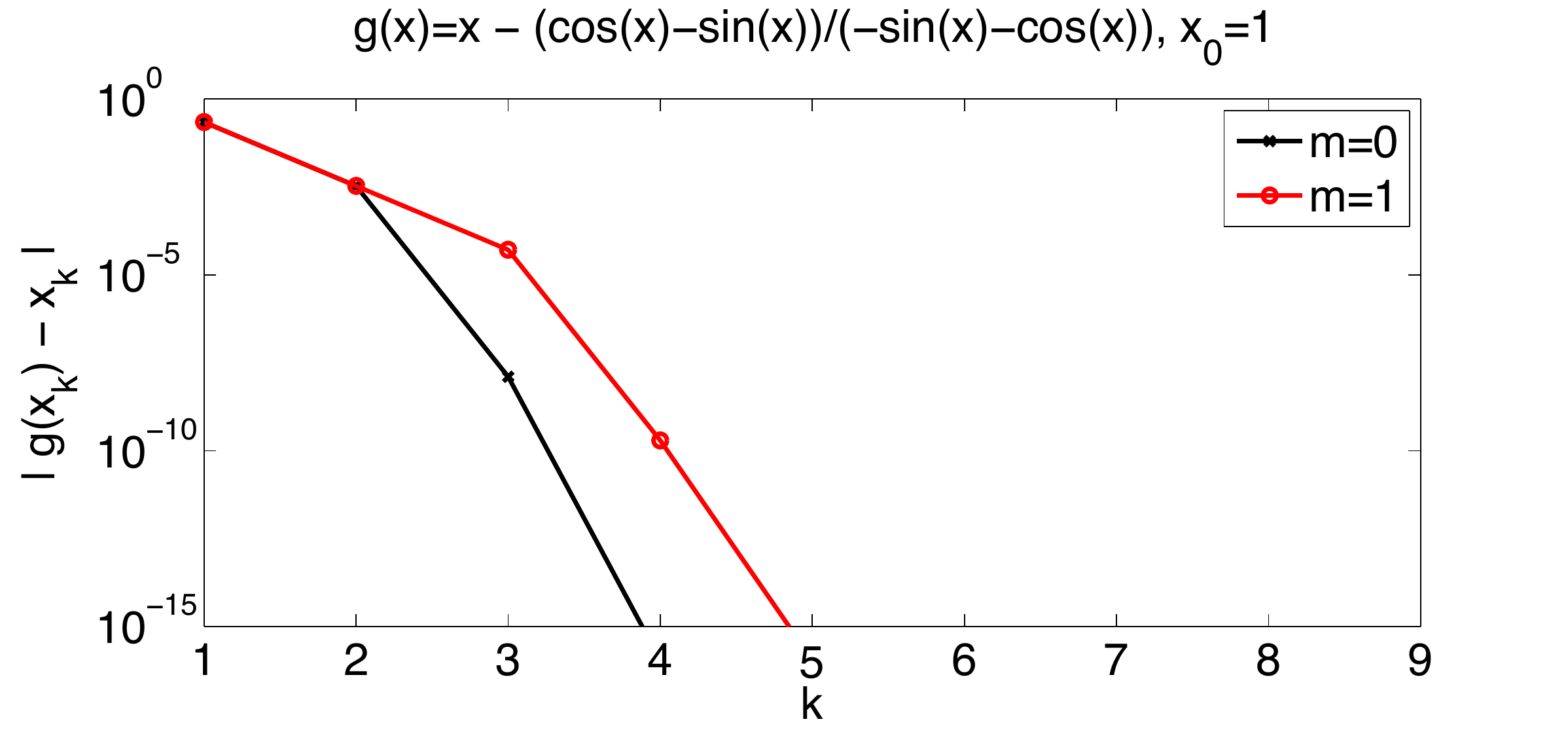}	\\
	\includegraphics[width = .48\textwidth,height=.25\textwidth]{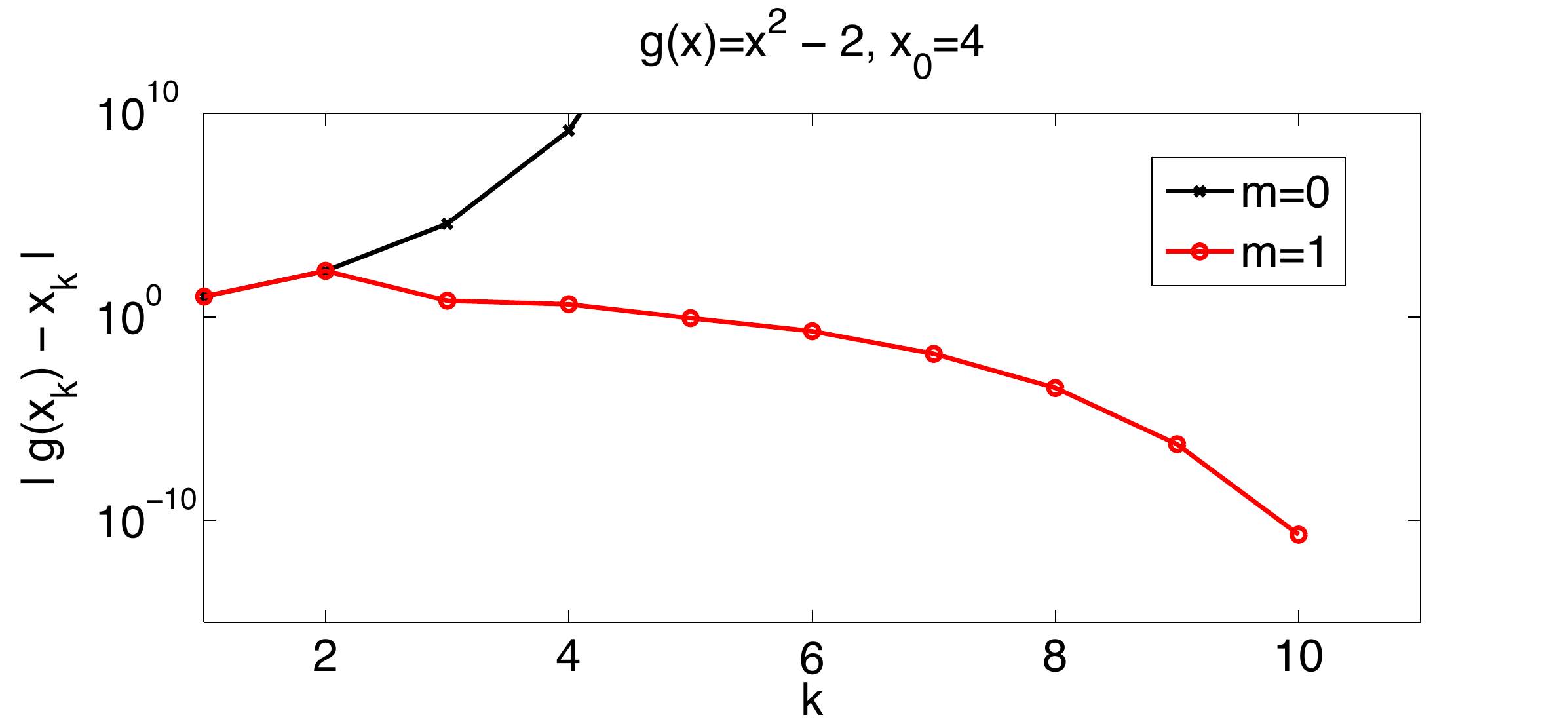}	
	\caption{\label{tests1Dc} Shown above are the residuals for three scalar fixed point iterations, with and without Anderson acceleration.}
\end{figure}

\subsection{Numerical tests for steady incompressible Navier-Stokes equation}
\label{subsec:nse}
Here we present numerical experiments to show the improved convergence provided by  
Anderson acceleration for solving the steady incompressible Navier-Stokes equations (NSE), which 
are given in a domain $\Omega$ by
\begin{eqnarray}
u \cdot\nabla u + \nabla p - \nu\Delta u  & = & f, \label{ns1} \\
\nabla \cdot u & = & 0, \label{ns2} 
\end{eqnarray}
where $\nu$ is the kinematic viscosity, $f$ is a forcing, $u$ and $p$ represent velocity and pressure, and the system
must be equipped with appropriate boundary conditions.  The $L^2(\Omega)$ norm and inner product will be denoted by $\|\cdot\|$ 
and $\forma$ in this subsection.

The tests we consider are for the 2D lid-driven cavity problem, which uses a domain $\Omega=(0,1)^2$, no slip ($u=0$) boundary conditions on the sides and bottom, and a `moving lid' on top which is implemented by the Dirichlet boundary condition $u(x,1)=\langle 1,0 \rangle^T$.  There is no forcing ($f=0$), and the kinematic viscosity is set to be $\nu \coloneqq Re^{-1}$, where $Re$ is the Reynolds number, and in our tests we use $Re$ varying between $1000$ and $10,000$.  Plots of the velocity streamlines for the steady NSE at $Re=2500$ and $6000$  are shown Figure \ref{2DCav2s}.

\begin{figure}[H]
\begin{center}
\ \ \ \ $Re$=2500 \hspace{1.4in} $Re$=6000  \\
\includegraphics[width = .32\textwidth, height=.28\textwidth,viewport=50 0 470 395, clip]{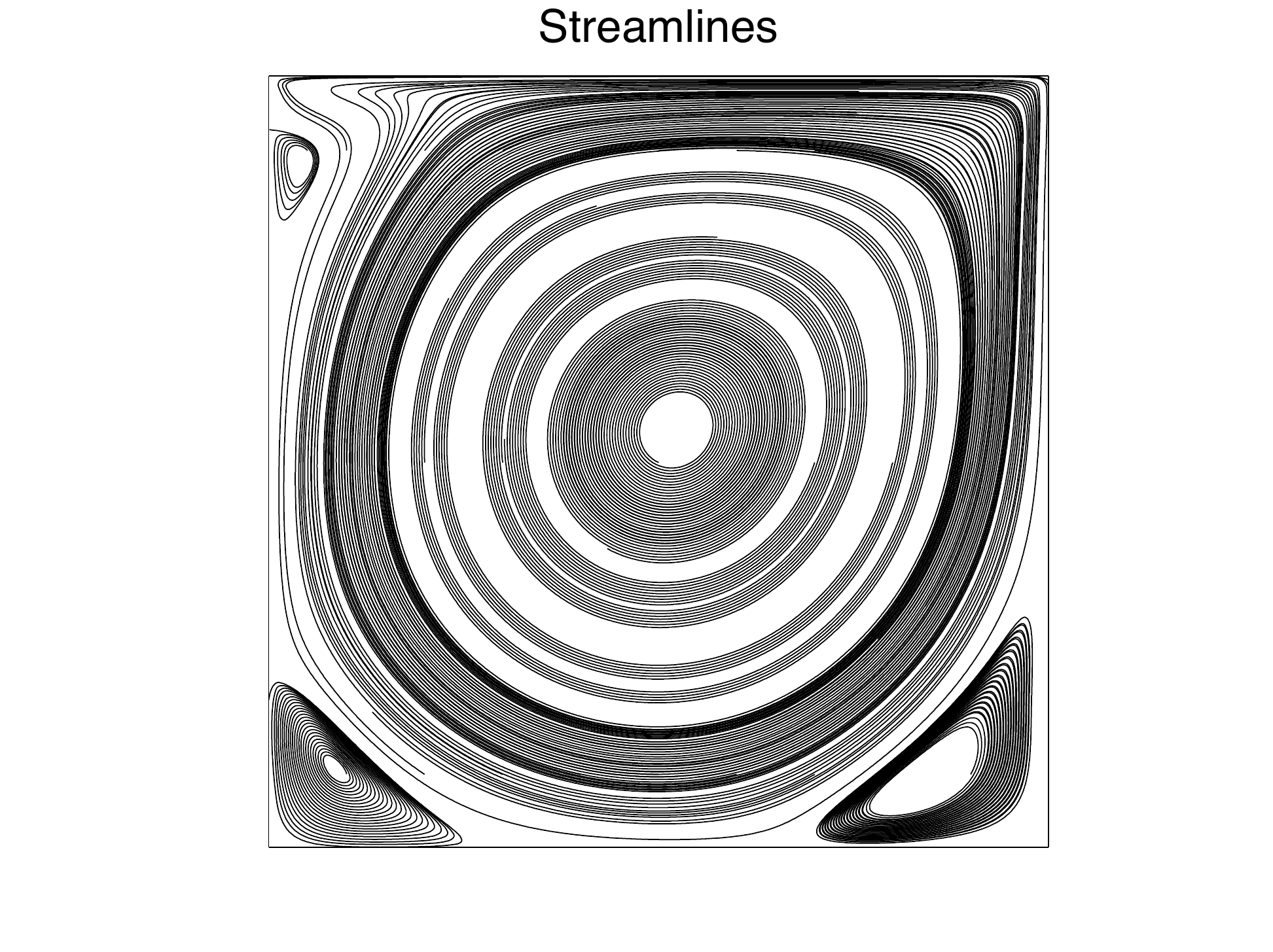}
\includegraphics[width = .32\textwidth, height=.28\textwidth,viewport=50 0 470 395, clip]{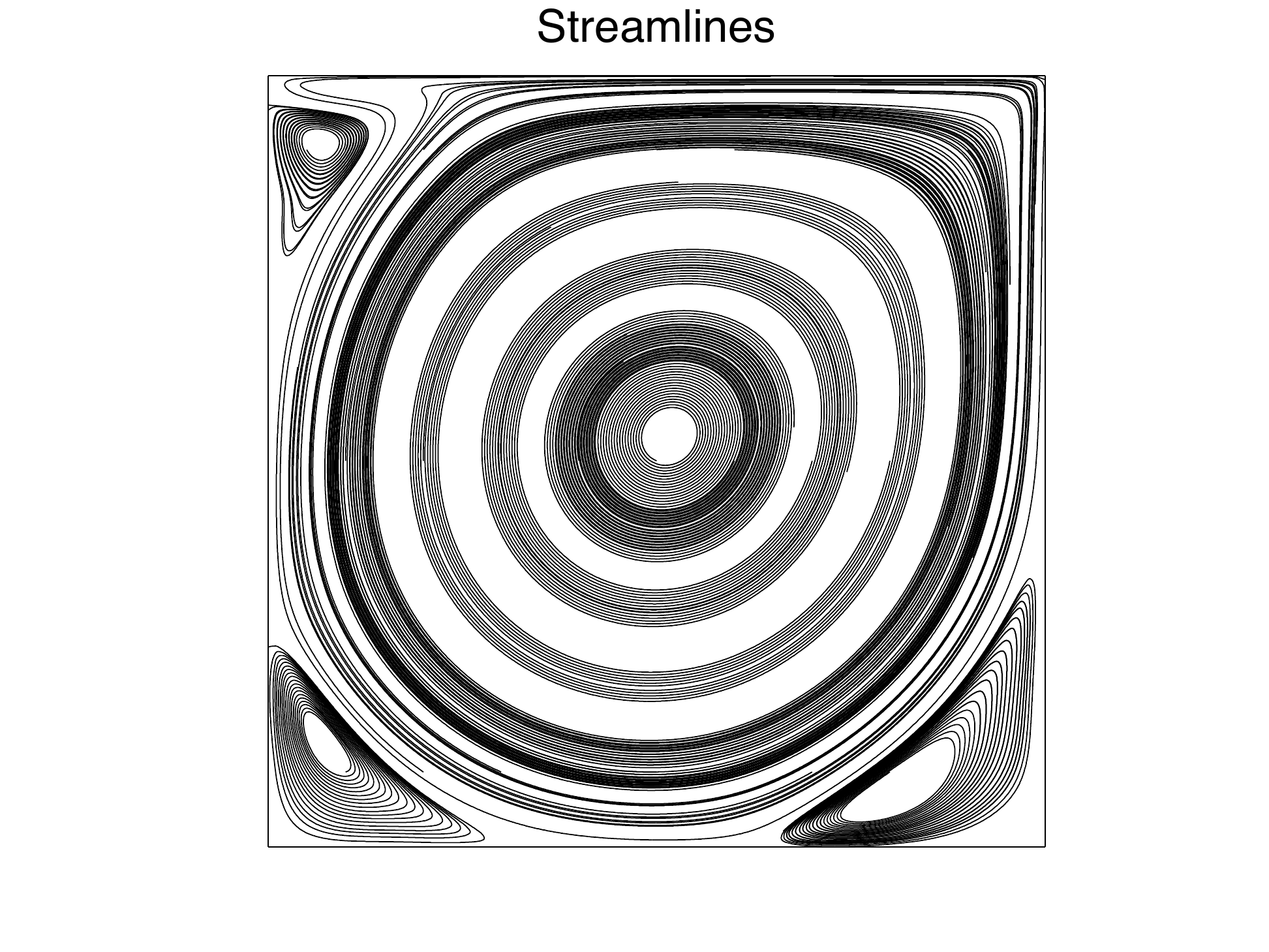}
	\caption{\label{2DCav2s} Streamline plots of the steady NSE driven cavity solutions with varying $Re$.}
	\end{center}
\end{figure}

We discretize with $(X_h,Q_h)=((P_2)^2,P_1)$ Taylor-Hood finite elements on a $\frac{1}{256}$ uniform triangular mesh that provides 592,387 total degrees of freedom, and for the initial guess we used $u_h^0=0$ but satisfying the boundary conditions. 
Define the trilinear form $b^*$ by
\[
b^*(u,v,w):= (u\cdot\nabla v,w) + \frac12 ((\nabla \cdot u)v,w).
\]  

The discrete steady incompressible NSE problem (with skew-symmetrized nonlinear term) reads as follows:  
Find $(u,p)\in (X_h,Q_h)$ satisfying for all $(v,q)\in (X_h,Q_h)$,
\begin{eqnarray}
- (p,\nabla \cdot v) + \nu(\nabla u,\nabla v) +b^*(u,u,v)
& = & (f,v) , \label{eqn:dnse1} \\
(\nabla \cdot u,q) & = & 0. \label{eqn:dnse2}
\end{eqnarray}
Since this problem is nonlinear, we need a nonlinear solver.  
We consider two common nonlinear iterations, Picard and Newton, which are defined as follows.

\begin{alg}[Picard iteration for steady NSE]\label{alg:usualPicard} \ \\
Step 1: Choose $u_0\in X_h.$\\
Step $k$: Find $(u_k,p_k)\in (X_h,Q_h)$ satisfying for all $(v,q)\in (X_h,Q_h)$,
\begin{eqnarray} \label{eqn:pnse1}
b^*(u_{k-1},u_k,v) - (p_k,\nabla \cdot v) + \nu(\nabla u_k,\nabla v)  & = &  ( f,v ), \\
(\nabla \cdot u_k,q) & = & 0. \label{eqn:pnse2}
\end{eqnarray}
\end{alg}

\begin{alg}[Newton iteration for steady NSE]\label{alg:usualNewton} \ \\
Step 1: Choose $u_0\in X_h.$\\
Step $k$: Find $(u_k,p_k)\in (X_h,Q_h)$ satisfying for all $(v,q)\in (X_h,Q_h)$,
\begin{eqnarray} \label{eqn:nnse1}
b^*(u_{k-1},u_k,v) + b^*(u_{k},u_{k-1},v) - b^*(u_{k-1},u_{k-1},v) - (p_k,\nabla \cdot v) + \nu(\nabla u_k,\nabla v)  & = & (f,v), \\
(\nabla \cdot u_k,q) & = & 0. \label{eqn:nnse2}
\end{eqnarray}
\end{alg}

For sufficiently small data, the steady NSE and these iterations are well-posed \cite{laytonbook}.  Hence we can consider both the Picard and Newton iterations as fixed point iterations $u_{k+1}=g(u_k)$, where $g$ is a solution operator of \eqref{eqn:pnse1}-\eqref{eqn:pnse2} for Picard or \eqref{eqn:nnse1}-\eqref{eqn:nnse2} for Newton.  In this way, we can apply Anderson acceleration to both methods.  Below, we test both the Picard and Newton iterations with Anderson acceleration 
(but note that we apply only the basic Picard and Newton methods, i.e. without relaxation or other variation that can aid in convergence).  The linear systems are solved with a sparse direct solver.

For Picard iterations, we observe in Figure \ref{2DCav2} (left side) that Picard without acceleration is converging linearly, although slowly; after 40 iterations, the residual is still $O(10^{-4})$.  Anderson acceleration makes a very significant improvement in the Picard convergence, with big improvement offered by $m=1$ and $m=2$, and even more by $m=3$.  With $m=3$ the residual after 40 iterations is about $O(10^{-9})$, and it would take usual Picard about another 50 iterations to reach this level for its residual.  

On the right side of Figure \ref{2DCav2}, we display the convergence behavior of the Newton iterations.  We observe the usual Newton iteration diverges, but with Anderson acceleration it converges for each of $m=1,\ 2, \ 3$.  This is an example of Anderson acceleration increasing the radius of convergence of a fixed point iteration.  The $m=1$ Anderson accelerated Newton iteration with $m=1$ achieves a residual of $10^{-14}$ after just 13 iterations.
It is important to note that such an improvement with small $m$ is also observed in \cite{LWWY12}.

\begin{figure}[H]
\begin{center}
\includegraphics[width = .4\textwidth, height=.3\textwidth,viewport=0 0 600 395, clip]
{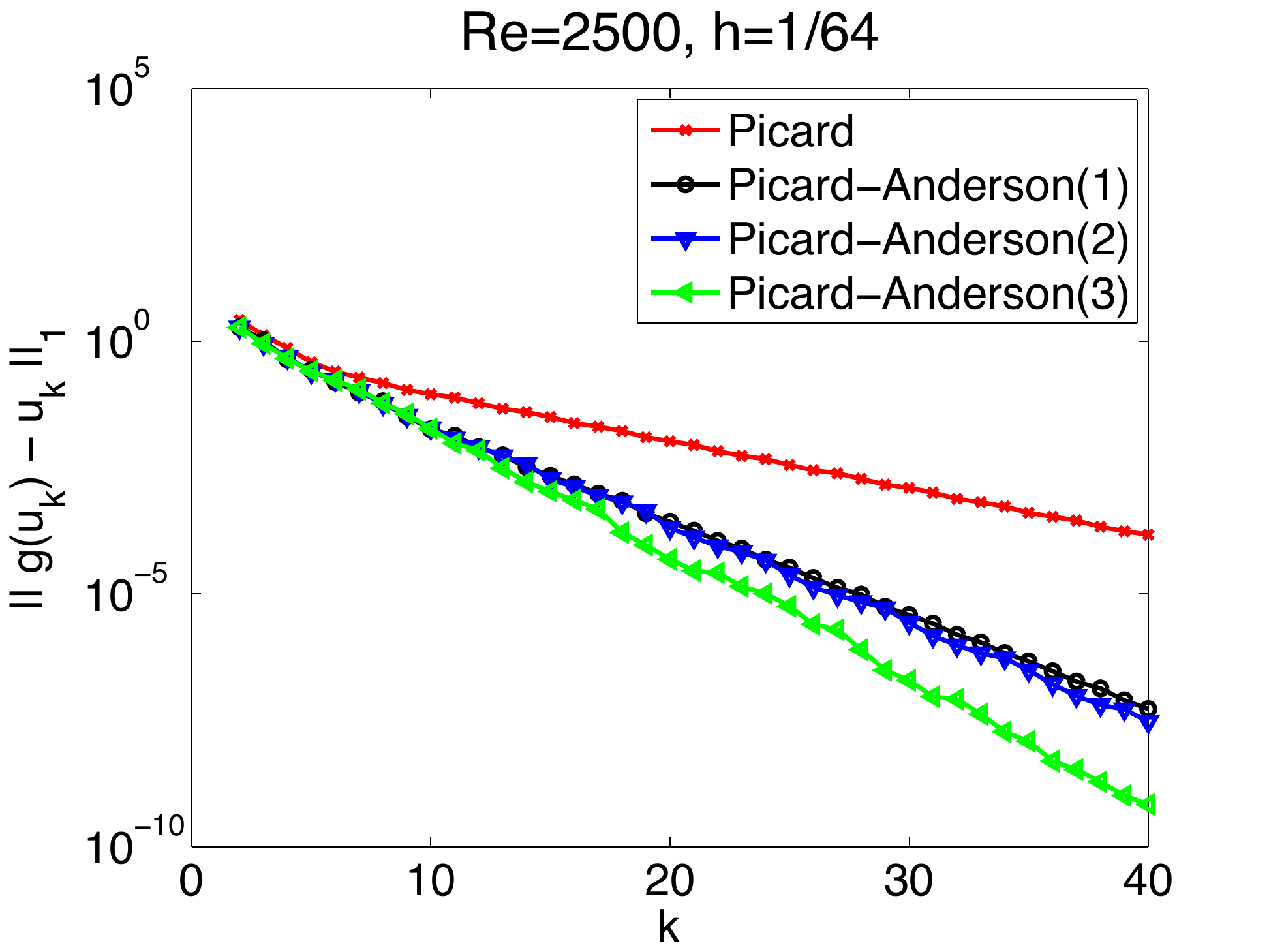}
\includegraphics[width = .4\textwidth, height=.3\textwidth,viewport=0 0 600 395, clip]
{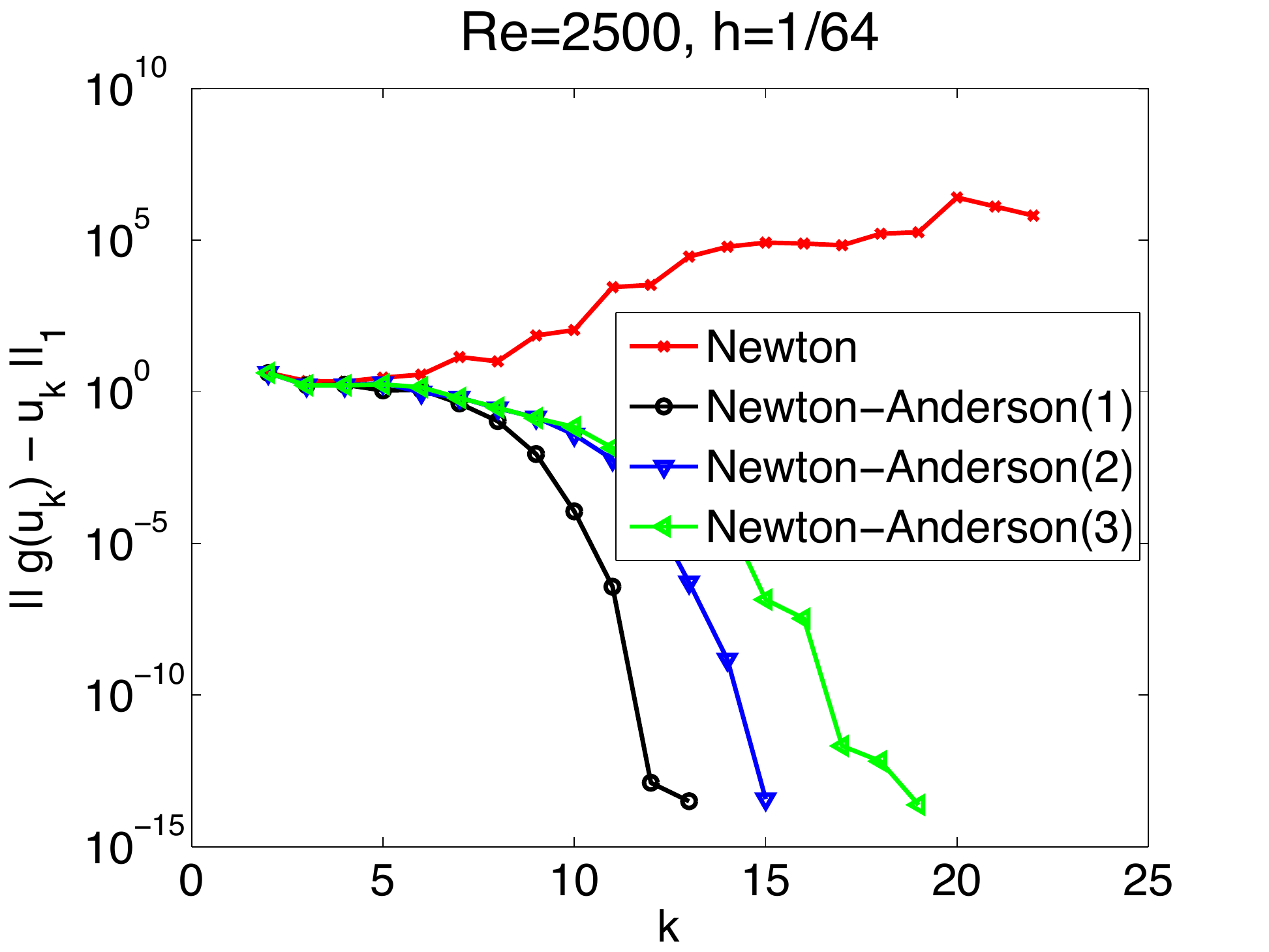}
\caption{\label{2DCav2} 
Convergence of the Anderson accelerated Picard and Newton iterations with $Re=2500$.}
\end{center}
\end{figure}
Results for $Re=6000$ are shown in Figure \ref{2DCav3}.  The usual Picard iteration fails here, as the residual over the last 20 iterations grows (although slightly), so $\kappa>1$ in this case.  Anderson acceleration helps Picard significantly, and will allow for convergence.  Usual Newton and $m=1,2$ Anderson-accelerated Newton iterations all failed (diverged), and we do not show these results in the plot.  The Anderson-accelerated Newton iteration with $m=3$ converged, and quite rapidly, reaching a residual of $O(10^{-14})$ in just 23 iterations.  
\begin{figure}[H]
\begin{center}
\includegraphics[width = .6\textwidth, height=.3\textwidth,viewport=0 0 900 345, clip]{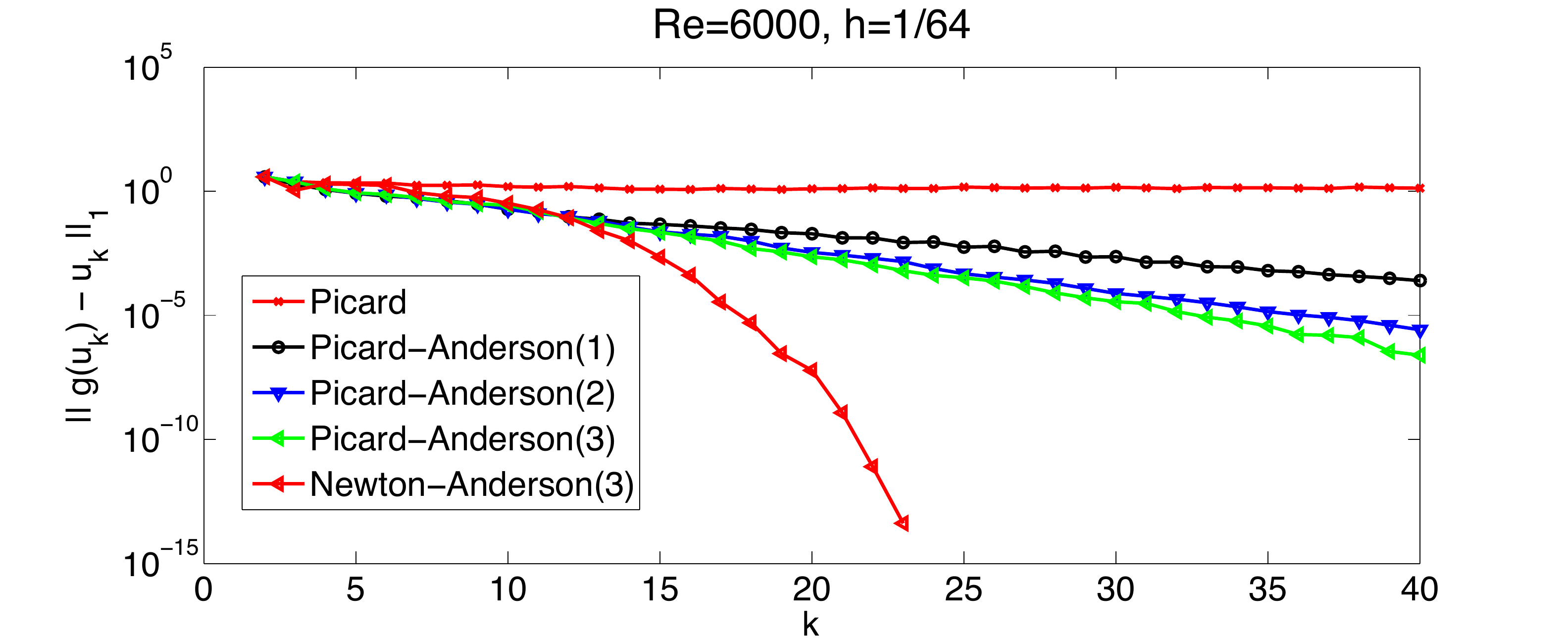}
\caption{\label{2DCav3} Convergence of the Anderson accelerated Picard and Newton iterations with $Re=6000$.}
\end{center}
\end{figure}
Box-plots of the $\theta_k$'s from the Picard iterations are shown in figure \ref{2DCav4}. For $Re=2500$ (left side), there is a clear
decreasing trend in distribution of $\theta$'s as $m$ increases, while for $Re=6000$ the boxplots look rather similar but with $m=1$ seemingly a little lower overall compared to $m=2$.  However, the lower values and outliers in these plots are critical, since one multiplication of a small factor takes many multiplications of larger factors to achieve the same residual decrease.
\begin{figure}[H]
\begin{center}
\includegraphics[width = .48\textwidth, height=.4\textwidth,viewport=0 0 400 345, clip]{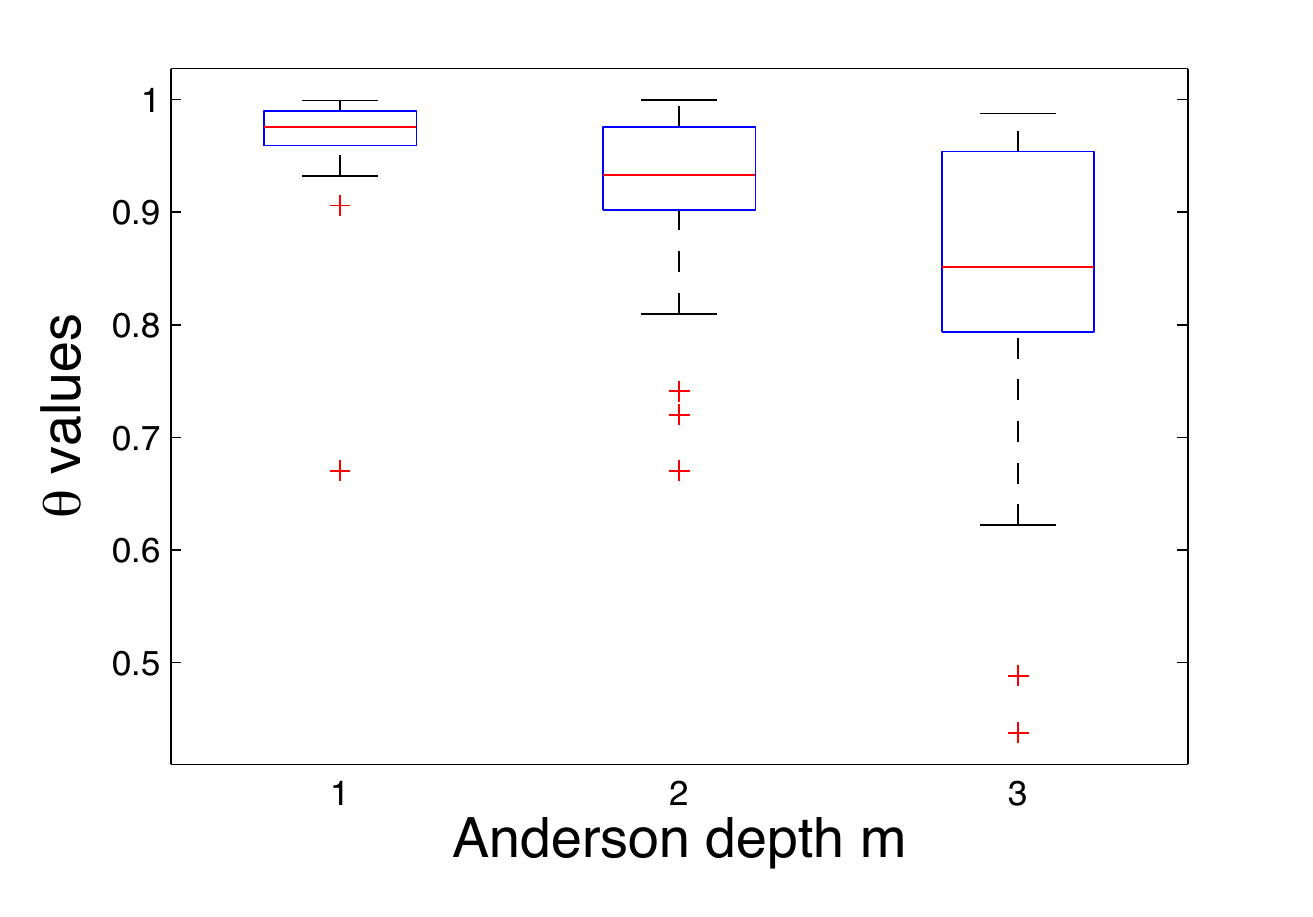}
\includegraphics[width = .48\textwidth, height=.4\textwidth,viewport=0 0 400 345, clip]{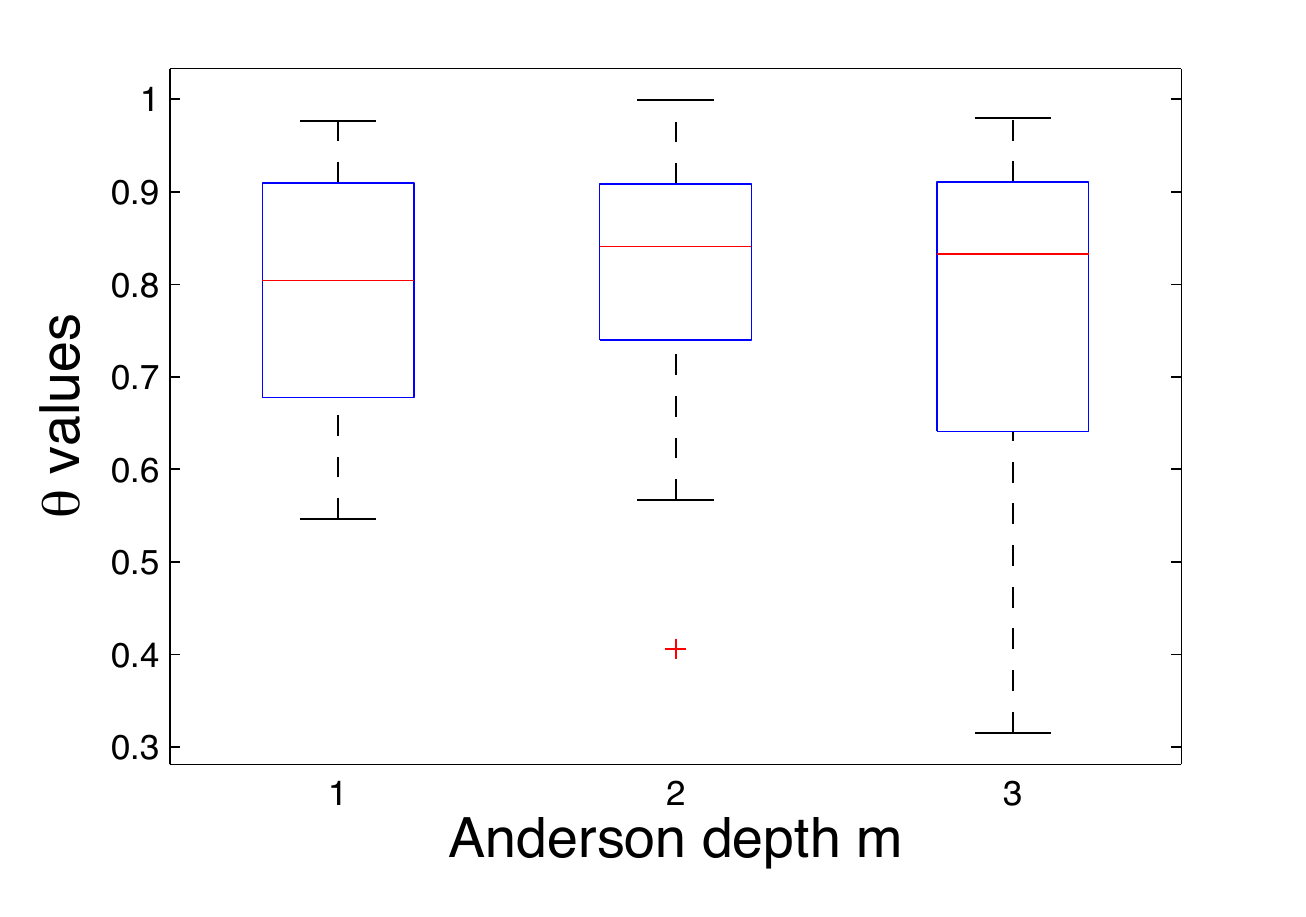}
	\caption{\label{2DCav4} Box-plots of $\theta$ values for the Picard iterations with $Re=2500$ (left) and $Re=6000$ (right).}
	\end{center}
\end{figure}

As a final part of this test, we compare 
the number of iterations needed to converge the residual for the Picard iteration in the $H^1$ norm to a tolerance
of $10^{-8}$, for varying $Re$ and varying $m$.  Results are shown in Table \ref{ReTable}, and again we observe a dramatic improvement from Anderson acceleration.  Even $m=1$ is enough to provide convergence up to $Re=10000$, although additional gain is made by increasing to $m=2$ and to $m=3$. It is interesting that convergence of the steady NSE is achieved for $Re=9000$ and $10000$ since the bifurcation point transition to transient flow is around $8000$ \cite{APQ02}, and thus the solutions found are seemingly unstable steady NSE solutions.  F denotes failure in the table, which we define as not converging within 500 iterations (but we note that inspection of the last few iterations of each of these that failed indicates the iterations are nowhere near, or even approaching, convergence).
 
\begin{table}[h!]
\begin{center}
\begin{tabular}{ccccc}
Re / m  &  0 & 1 & 2 & 3  \\ \hline
1000 & 36 & 32 & 29 & 26 \\ \hline
2000 & 48 & 41 &  40   & 34 \\ \hline
3000 & 86 & 49 & 45 & 37 \\ \hline
4000 & 158 & 59 & 46 & 40 \\ \hline
5000 & 363 & 55 & 48 &  44  \\ \hline
6000 & F & 62 & 55  & 49   \\ \hline 
7000 & F & 65 & 61  &  53 \\ \hline 
8000 & F & 78 & 70  & 58 \\ \hline 
9000 & F & 94 & 83  & 68 \\ \hline 
10000 & F & 105 & 97  & 71 \\ \hline 
\end{tabular}
\caption{ \label{ReTable} Shown above are the number of Picard iterations needed to converge the nonlinear residual for the steady
NSE up to $10^{-8}$ in the $H^1$ norm, for varying $Re$ and $m$.  F denotes a failure to reach convergence by 500 iterations.}
\end{center}
\end{table}

\subsection{Damping tests with a quasilinear equation}\label{subsec:qldamp}
The damping parameter $\beta$ of Algorithm \ref{alg:anderson} may become important
for convergence 
in the case of a fixed-point operator $g$ that is not
contractive.  A simple example of this type of problem is the quasilinear equation
$-\div(a(u) \grad u) = f$ in a domain $\Omega$ with homogeneous Dirichlet boundary
conditions.  In weak form
\begin{align}\label{eqn:ql-ex}
(a(u) \grad u, \grad v) = (f,v),
\end{align}
where $\forma$ denotes the $L^2$ inner product as in the above example.  
This can be thought of as a simple model of the effective nonlinearity in a 
steady Richards' type equation modeling the pressure $u$ in partially saturated media,
where $a(u)$ is the hydraulic conductivity which depends nonlinearly on the pressure
head via the saturation.
In this example we take $\Omega = (0,1)$ and
\[
a(u) = k + \tanh((u-u_0)/\eps), ~\text{ with } u_0 = 0.5, \quad k = 1.01, 
~\text{ and } \eps = 0.1.
\]
The function $f$ is chosen so the exact solution is $u^\ast = 10\sin(\pi x)$.
For the results below, the 1D problem is discretized with piecewise linear (P1)
finite elements with a uniform meshsize of $h=1/16384$.
For this example $m_k = 0$ for $k<m$ and $m$ otherwise.  The optimization problem
is solved with an economy $QR$ decomposition and $\theta_k$ is computed as described in
\S \ref{subsec:optgain}. The fixed-point operator $\tilde u^{k+1}= g(u^k)$ solves 
$(a(u^k) \grad g(u^k), \grad v) = (f,v)$, as in a basic Picard iteration.

As seen by the expansion \eqref{eqn:wk008}, 
the results of Theorems \ref{thm:m1} and \ref{thm:m2} as well as
Proposition \ref{prop:gm}, the damping factor $\beta_{k-1}$ affects the first order term
$\theta_k(1-\beta_{k-1} + \kappa\beta_{k-1})\nr{w_k}$,
but not the higher order terms. If the operator $g: X \goto X$ is not assumed 
contractive, then Assumption \ref{assume:contract} does not hold, and 
\eqref{eqn:wk008} then provides a blueprint for bounding $\nr{w_{k+1}}$ 
by $\nr{w_k}$ and higher-order terms involving differences of consecutive iterates
\[
\nr{w_{k+1}} \le \theta_k(1-\beta_{k-1} + \kappa\beta_{k-1})\nr{w_k}
+ \bigo\left(\nr{e_k}^2 \right) + \ldots + \bigo\left(\nr{e_{k-m}}^2\right),
\]
as the bounds of \S \ref{subsec:ew} controlling the difference between consecutive
iterates by the residuals do not hold in the noncontractive setting.  
It is remarked however that in the  
contractive setting of \S \ref{subsec:ew}, the control of the 
error terms $\nr{e_j}$ in terms of residuals $\nr{w_n}$ is independent of the damping.

This first order effect of the damping agrees with that seen for the error in the 
fixed-point iteration alone.  If the update step of the damped fixed-point iteration 
for operator $g$ with fixed-point $x^\ast$ is given by
$u_{k+1} = (1-\beta) u_k + \beta g(u_k)$ then
\[
u_{k+1} - u^\ast = (1-\beta)(u_k - u^\ast) + \beta(g(u_k) - g(u^\ast))
= (1-\beta)(u_k - u^\ast) + \beta g'(z_k^\ast(t);u_k - u^\ast),
\]
with  $z_k^\ast(t) = u^\ast + t(u_k - u^\ast)$.  
As this example is easily seen to satisfy Assumption \ref{assume:g},
 this immediately yields the norm-bound
$\nr{u_{k+1}-u^\ast} \le ((1-\beta) + \kappa \beta)\nr{u_k - u^\ast}$.

If $g$ is contractive then $\beta = 1$ (no damping) gives the best convergence rate.  
In this example however, Assumption \ref{assume:contract} does not hold globally. 
For instance near the boundary $u$ approaches zero and  
$k + (\tanh(u-u_0)/\eps)$ is close to $k-1 = 10^{-2}$, and the locally 
small ellipticity coefficient can cause failure of the method to converge.
This is demonstrated in the first plot of Figure \ref{fig:qldamp-smallm} 
where on the left the fixed-point iteration fails to converge to a tolerance of 
$10^{-5}$ with $\beta = \{1,0.8,0.6\}$, although more accuracy is attained with the 
damped iterations. It is also clear from this first plot that in the regime where
the the operator $g$ is contractive (the beginning of the calculation), the damping
has the predicted linear effect on the convergence rate.

The second and third plots of Figure \ref{fig:qldamp-smallm} show the effect of 
damping factors $\beta =\{1.0, 0.8, 0.6\}$ as well as an adaptive strategy for the
cases of Anderson depths $m=1$ and $m=2$.  The adaptive strategy is based on the
convergence rates found in Theorems \ref{thm:m1} - \ref{thm:m2} and 
Proposition \ref{prop:gm}, meaning $\beta$ plays an active role in decreasing the 
coefficient of the first order term particularly when $\theta$ is not small enough.
So $\beta_{adapt} = 1- \theta_k/2$ is chosen as a simple heuristic to set a factor
between $0.5$ and $1.0$ that is close to unity when $\theta_k$ is small and 
approaches $0.5$ as $\theta$ approaches one.

\begin{figure}[H]
\begin{center}
\includegraphics[width = .32\textwidth]{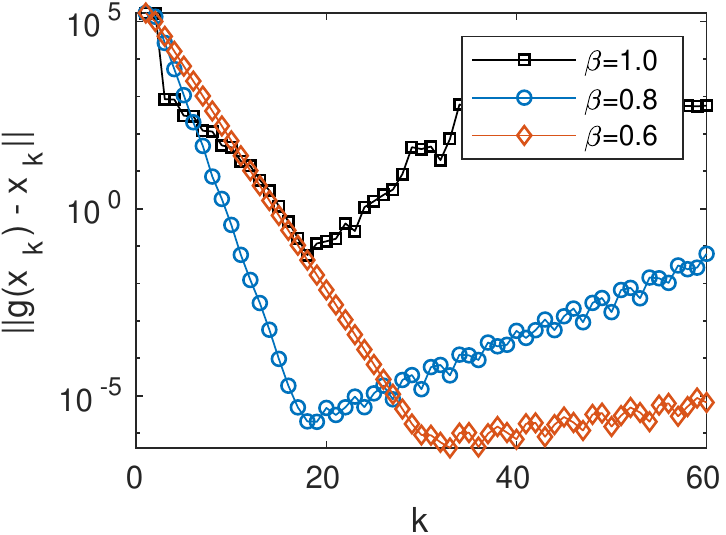}
~\includegraphics[width = .32\textwidth]{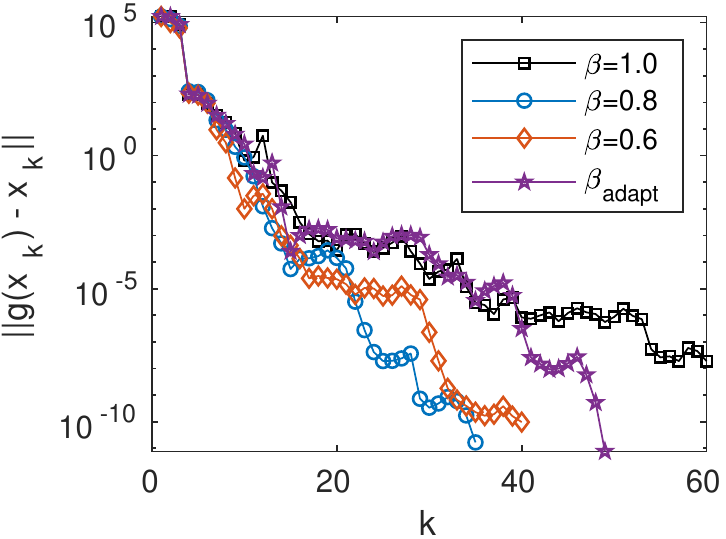}
~\includegraphics[width = .32\textwidth]{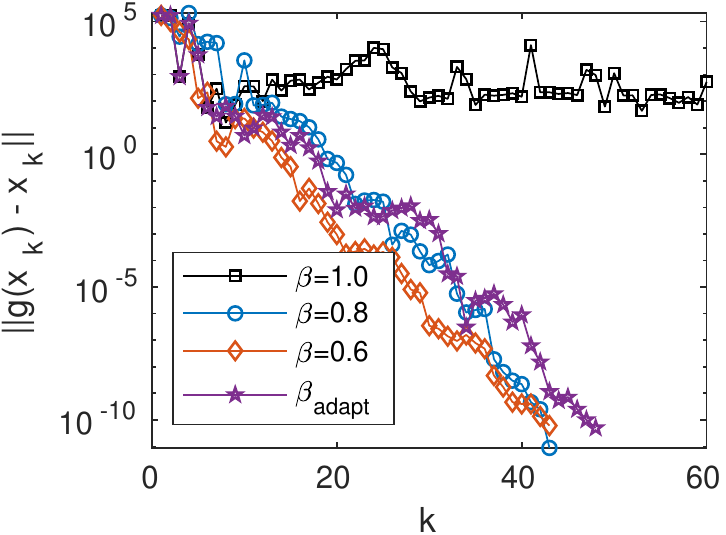}
\caption{\label{fig:qldamp-smallm} 
Left: Damped iterations for \eqref{eqn:ql-ex}
with $m=0$. Center: Damped iterations with $m=1$. Right: 
Damped iteration with $m=2$. }
\end{center}
\end{figure}

The three plots of Figure \ref{fig:qldamp-largem} 
illustrate the behavior of the same damping factors for 
greater Anderson depths, $m = \{4, 6, 8\}$. While the three examples in
Figure \ref{fig:qldamp-smallm} failed to converge without damping, the three examples
for greater depth $m$ in Figure \ref{fig:qldamp-largem} converged both with and 
without, but generally better with some damping.  The adaptive strategy, while not 
optimal, demonstrates proof of concept that with the gain $\theta_k$ taken into 
account, damping can be designed without extensive
experimentation or additional computation to stabilize the convergence for difficult 
problems. 
\begin{figure}[H]
\begin{center}
\includegraphics[width = .32\textwidth]{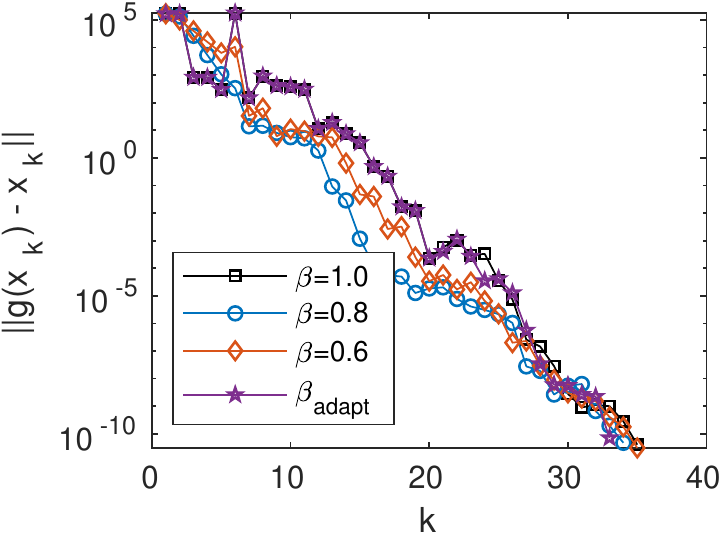}
~\includegraphics[width = .32\textwidth]{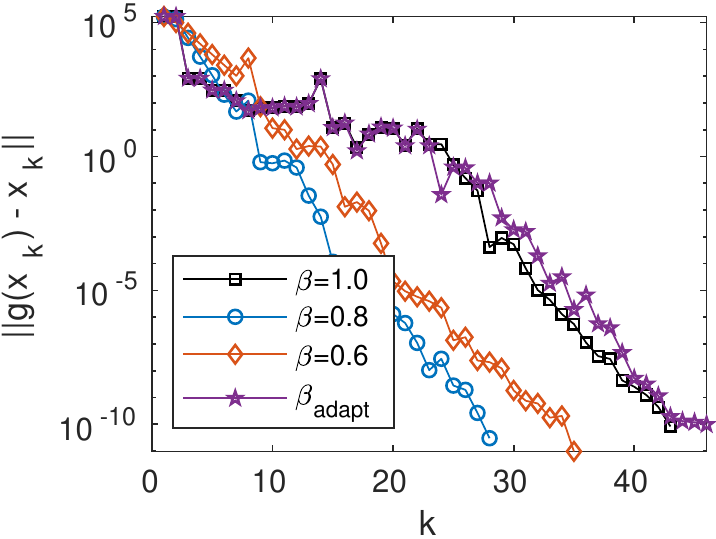}
~\includegraphics[width = .32\textwidth]{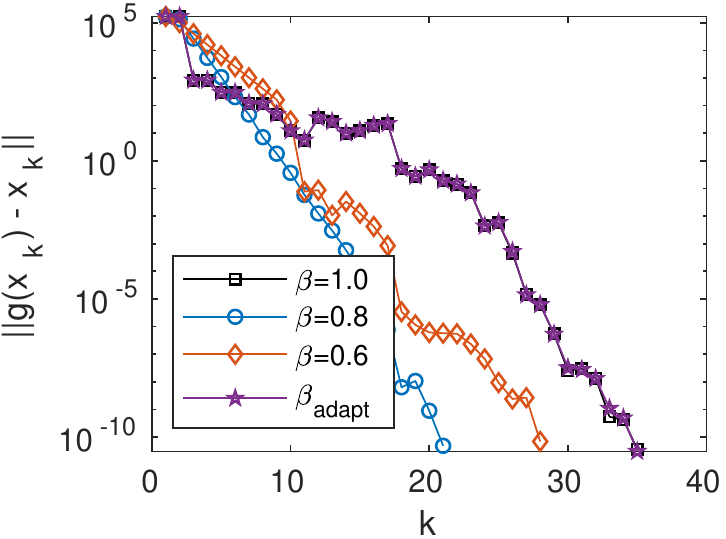}
\caption{\label{fig:qldamp-largem} 
Left: Damped iterations for \eqref{eqn:ql-ex}
with $m=4$. Center: Damped iterations with $m=6$. Right: 
Damped iteration with $m=8$.}
\end{center}
\end{figure}

\section{Conclusion}\label{sec:conc}
We have proven that Anderson acceleration improves the
first-order convergence rate for fixed point iterations, in agreement with 
decades of experimental results.  
We show that the increase in the linear convergence rate at each step depends
on the gain from the optimization step, but that additional quadratic error terms arise.  Hence as long as the gain from the 
optimization stage dominates these quadratic error terms, the convergence rate will be increased.  In particular for 
linearly convergent fixed point methods, an improved convergence rate from Anderson acceleration is expected; however,
for methods converging quadratically, the convergence will typically be slightly slowed.
Additionally, our results provide justification that both increasing 
the depth $m$ and using damping increases the radius of convergence.  
Results of numerical tests have been provided to illustrate our theory.


\end{document}